\documentclass[11pt,letterpaper]{article}
\usepackage[margin=1in]{geometry}
\usepackage[T1]{fontenc}
\usepackage{microtype}
\usepackage{graphicx}
\usepackage{amsmath,amssymb,amsfonts,bm}
\usepackage{amsthm}
\usepackage{xcolor}
\usepackage{array,multirow,makecell,booktabs}
\usepackage{enumitem}
\usepackage{placeins}
\usepackage{algorithm,algpseudocode}
\usepackage{comment}
\usepackage[authoryear,round]{natbib}
\usepackage[font=small,labelfont=bf]{caption}
\usepackage{etoolbox}
\AtBeginEnvironment{table}{\small}
\usepackage[hidelinks]{hyperref}
\usepackage[nameinlink,noabbrev]{cleveref}
\hypersetup{hidelinks,hypertexnames=false,bookmarksdepth=2}

\theoremstyle{plain}
\newtheorem{thm}{Theorem}
\newtheorem{lem}{Lemma}
\newtheorem{prop}{Proposition}
\newtheorem{cor}{Corollary}
\theoremstyle{definition}
\newtheorem{defn}{Definition}
\newtheorem{assumption}{Assumption}
\theoremstyle{remark}
\newtheorem{rem}{Remark}
\numberwithin{equation}{section}
\crefname{thm}{theorem}{theorems}
\Crefname{thm}{Theorem}{Theorems}

\crefname{lem}{lemma}{lemmas}
\Crefname{lem}{Lemma}{Lemmas}

\crefname{prop}{proposition}{propositions}
\Crefname{prop}{Proposition}{Propositions}

\crefname{cor}{corollary}{corollaries}
\Crefname{cor}{Corollary}{Corollaries}

\crefname{defn}{definition}{definitions}
\Crefname{defn}{Definition}{Definitions}

\crefname{assumption}{assumption}{assumptions}
\Crefname{assumption}{Assumption}{Assumptions}

\crefname{rem}{remark}{remarks}
\Crefname{rem}{Remark}{Remarks}

\crefname{equation}{equation}{equations}
\Crefname{equation}{Equation}{Equations}

\crefname{section}{section}{sections}
\Crefname{section}{Section}{Sections}

\crefname{figure}{figure}{figures}
\Crefname{figure}{Figure}{Figures}

\crefname{table}{table}{tables}
\Crefname{table}{Table}{Tables}

\crefname{algorithm}{algorithm}{algorithms}
\Crefname{algorithm}{Algorithm}{Algorithms}

\newcommand{\R}{\mathbb{R}}

\newcommand{\cA}{\mathcal{A}}

\newcommand{\de}{\mathrm{d}}
\DeclareMathOperator*{\argmin}{arg\,min}
\newcommand{\SweepN}{2400}
\newcommand{\SweepZero}{78.5}
\newcommand{\SweepAboveOne}{10.9}
\newcommand{\SweepAboveTen}{2.9}
\newcommand{\SweepMax}{26.3}
\newcommand{\SweepMaxAbs}{0.83}

\begin{document}

\title{Time-optimal vaccination in state-dependent SIR models:
a single switch criterion and numerical examples outside its regime}

\author{Yeachan Park$^{1}$\thanks{Yeachan Park and Jiwoong Jang contributed equally to this work.}
\quad Jiwoong Jang$^{2}$\footnotemark[1]
\quad Yeoneung Kim$^{3}$\thanks{Corresponding author: \href{mailto:yeoneung@yonsei.ac.kr}{\texttt{yeoneung@yonsei.ac.kr}}.}\\[0.7em]
\small $^{1}$Department of Mathematics, Sejong University, Seoul, Republic of Korea\\
\small $^{2}$Department of Mathematics, University of Maryland-College Park,\\
\small William E. Kirwan Hall, 4176 Campus Drive, College Park, Maryland 20742, USA\\
\small $^{3}$Department of Industrial Engineering, Yonsei University, Seoul, Republic of Korea}
\date{}
\maketitle

\begin{abstract}
In a susceptible, infected, and recovered (SIR) model with constant rates, a
time-optimal vaccination policy has at most one switch from no vaccination to
maximal vaccination.  We allow the transmission and recovery rates
$\beta(S,I)$ and $\gamma(S,I)$ to depend on the state.  If
$\beta+I\partial_I\beta>0$ and
$G:=\partial_S(\beta S-\gamma)>0$ throughout the forward invariant state
space, then every optimal control is bang-bang with at most one switch.
The incidence condition allows the transmission coefficient to decrease with
prevalence.  The Pontryagin analysis shows that $G$, the sensitivity of per
capita infected growth to the
susceptible fraction, and the infected costate together determine the
crossing direction of the switching function.  We also prove existence of an
optimal control, derive explicit a priori upper bounds on the eradication
time and characterize the switching time by an optimization over one
variable.  For parameter sets that violate $G>0$, numerical optimization
gives controls with two switches and smaller eradication times than the
computed minima over the single switch family.  These computations provide a
reproducible numerical counterexample to an unrestricted single switch claim.
In a prescribed
parameter sweep satisfying $G>0$, delayed vaccination reduces the eradication
time relative to immediate full vaccination by more than $1\%$ in
$\SweepAboveOne\%$ of the cases, with a maximum reduction of
$\SweepMax\%$.  In examples near this maximum, the delay increases peak prevalence
by more than an order of magnitude and cumulative incidence more than
fourfold.  Hence minimizing eradication time does not in general minimize
epidemiological burden.
\end{abstract}

\noindent\textbf{Keywords:} time-optimal vaccination, epidemic models, state-dependent SIR model, bang-bang control, Pontryagin maximum principle, dynamic programming

\medskip
\noindent\textbf{Mathematics Subject Classification (2020):} 49K15, 49L20, 92D30

\section{Introduction}\label{sec:intro}

Susceptible, infected, and recovered (SIR) models are widely used to study
epidemic dynamics and intervention strategies.  Since the work of
\citet{KermackMcKendrick1927}, these models have been analyzed extensively
because their compartments have direct interpretations and the governing
equations are mathematically tractable
\citep{AndersonMay1991,DiekmannHeesterbeek2000,Hethcote2000,Martcheva2015}.
Control variables can represent vaccination, treatment, or behavioral
interventions.  Applications of optimal control to epidemic models are
reviewed by \citet{Wickwire1977} and \citet{LenhartWorkman2007}.
Early studies obtained bang-bang isolation and immunisation policies for stochastic and deterministic epidemic models \citep{Abakuks1973,Abakuks1974,MortonWickwire1974}, and \citet{Behncke2000} provided a general analysis of deterministic epidemic control.
More recent work has addressed vaccination and treatment strategies
\citep{Zaman2008,GaffSchaefer2009}, singular arcs in SIR models with combined
controls \citep{LedzewiczSchattler2011}, control under limited resources
\citep{HansenDay2011,BolzoniLimitedResources2019}, and behaviourally coupled
vaccination programmes \citep{BuonomoDellaMarca2019,Buonomo2019}.
Healthcare capacity has been incorporated through costs for exceeding
available medical resources \citep{Ketcheson2021} or explicit ICU constraints
\citep{Miclo2022,DellaRossa2025}.

We consider the \emph{minimum eradication time} criterion, in which the
infected population is required to reach a prescribed threshold in minimum
time.  This criterion gives an optimal control problem with an exit time and
a free terminal time.
Such problems can be studied with Pontryagin's Maximum Principle
\citep{PontryaginEtAl1962,LeeMarkus1967,SchattlerLedzewicz2012} and dynamic programming
\citep{BardiCapuzzoDolcetta1997,FlemingSoner2006}.

Most analyses of eradication time in SIR models assume constant transmission
and recovery rates.
For the constant rate vaccination model, \citet{BolzoniTimeOptimal2017}
establish that a time-optimal policy has at most one switch from no
vaccination to maximal vaccination.
\citet{HyndEradication2021} develop the dynamic programming analysis and
characterize the value function as the unique continuous viscosity solution
of the associated HJB equation under their boundary convention.  They derive a scalar
switching time representation, an earliest value matching characterization,
and a derivative criterion for the switch.  For the uncontrolled SIR model,
\citet{HyndCriticalTimes2022} analyze the time at which prevalence begins to
decrease and the time at which it reaches a prescribed threshold.
Related bang-bang results hold for an SIR model with an Erlang distributed
infectious period.  \citet{BolzoniDellaMarca2025} split the infected
compartment into stages and minimize epidemic burden under limited
vaccination resources or linear vaccination costs.  They prove that optimal
controls have at most one switch, from maximal vaccination to no vaccination
when a switch occurs.  Their objective and switching direction thus differ
from the minimum time problem considered here.  Their rates remain constant,
whereas our model has two state variables and rates that depend on state.
The latter dependence introduces additional nonlinear terms into the adjoint
system.  Related work by \citet{JangKim2025} treats minimum eradication time
for an SIR model with coefficients that depend on time and establishes a
Hamilton--Jacobi characterization.  Coefficients that depend on state instead
enter the costate equations through the state trajectory.
The Hamilton--Jacobi equation for the constant rate model has also been solved
with physics informed neural networks
\citep{RaissiPerdikarisKarniadakis2019,KimKimKim2025}.  This gives an
approximation of the value function and the induced vaccination policy.

Constant transmission and recovery rates are restrictive for some epidemic
models.
Behavioral responses, saturation, and heterogeneous contacts can make
transmission depend on state.  Healthcare capacity and treatment availability
can similarly affect recovery.  Nonlinear incidence has been studied since
\citet{CapassoSerio1978}; see also
\citet{LiuHethcoteLevin1987,RuanWang2003,XiaoRuan2007}.
Contact and vaccination decisions that depend on prevalence are surveyed by
\citet{FunkSalatheJansen2010,ManfrediDOnofrio2013} and analysed by
\citet{dOnofrioManfrediSalinelli2007,Weitz2020}.  A removal rate limited by
hospital capacity can also change the bifurcation structure of the SIR model
\citep{ShanZhu2014}.
These mechanisms motivate controlled SIR models in which transmission and
recovery depend on state.  This dependence introduces additional
terms into the adjoint system, and the constant rate result does
not extend without further conditions.  The relevant quantity is
$G=\partial_S(\beta S-\gamma)$, the sensitivity of the instantaneous per
capita infected growth rate to the susceptible fraction.  We prove the single
switch result under $G>0$ and $\beta+I\partial_I\beta>0$.  The latter condition
requires incidence to increase with prevalence at fixed $S$, while allowing
the transmission coefficient itself to decrease.  Numerical examples
with $G<0$ over part of the state space yield controls with two switches and
provide reproducible numerical counterexamples to the unrestricted single
switch conclusion.

\medskip
\noindent\textbf{Contributions.}
The main contributions are as follows.

\begin{itemize}[leftmargin=16pt]
\item We prove that, under $\beta+I\partial_I\beta>0$ and
$G=\partial_S(\beta S-\gamma)>0$, every time-optimal vaccination policy is
bang-bang with
\emph{at most one} switching time, from no vaccination to full vaccination.
The proof uses Pontryagin's Maximum Principle and an
analysis of the switching function.  The crossing
direction of the switching function is governed by the product $\lambda_I G$,
where $G:=\partial_S(\beta S-\gamma)$ is the sensitivity of the per capita infected
growth rate to the susceptible fraction.  The terms that depend on prevalence,
$\partial_I\beta$ and $\partial_I\gamma$,
enter only through the adjoint dynamics of the infected costate, whose sign
structure follows from the terminal conditions and the incidence condition.
An analytic example includes transmission that decreases with prevalence.

\item For admissible parameters violating $G>0$, numerical optimization gives
a control that starts vaccination, pauses it, and resumes it.  Its computed
eradication time is smaller than the computed minimum over the $0\to1$
single switch family.

\item A parameter sweep quantifies the improvement over immediate full
vaccination.  The reduction exceeds $1\%$ in $\SweepAboveOne\%$ of the parameter sets
and does not exceed $\SweepMax\%$.  Immediate full vaccination is the computed
optimum in $\SweepZero\%$ of the sets.  In the cases with the largest reduction,
delay increases peak prevalence by more than an order of magnitude and
cumulative incidence more than fourfold.  These calculations quantify, for
the selected nonlinear families, the conflict between duration and burden
identified by \citet{BolzoniTimeOptimal2017}.
\end{itemize}

We also establish existence and uniqueness of state trajectories and
existence of a time-optimal control.  Explicit upper bounds on the
eradication time guarantee finite simulation horizons independently of the
switching analysis.  Under the structural criterion, the single switch
result gives the representation
\[
V(z)=\inf_{\tau\ge 0}\{\tau+W(\Phi^0_\tau(z))\},
\]
where $W$ is the exit time under full vaccination and $\Phi^0$ is the flow
without vaccination.  Together with the value matching condition and the
identity $\phi'(\tau)=S\,\partial_SW$, this extends the switching time
characterizations of \citet{HyndEradication2021} to the present class of
state dependent rates.

\medskip
\noindent\textbf{Organization.}
Section~\ref{sec:problem} introduces the controlled SIR model and its
assumptions.  Section~\ref{sec:PMP_switch} proves the bang-bang and single
switch properties.
Section~\ref{sec:bounds} derives explicit upper bounds on the eradication time.

Section~\ref{sec:DPP_switch} proves existence of state trajectories and
optimal controls, gives the dynamic programming representation, and derives
a value matching condition.
Section~\ref{sec:numerics} verifies the structural condition for three
representative nonlinear parametrizations, reports numerical experiments
that confirm the single switch prediction and compare the observed
eradication times against the explicit bounds, and examines numerically the
behavior when the structural condition is violated.

\section{Problem formulation}\label{sec:problem}

\subsection{Controlled SIR dynamics}

We consider the controlled SIR system
\begin{equation}\label{eq:SIR}
\begin{cases}
\dot S = -\beta(S,I)\,S\,I - u(t)\,S,\\
\dot I = \beta(S,I)\,S\,I - \gamma(S,I)\,I,
\end{cases}
\qquad u(t)\in[0,1],
\end{equation}
where $S(t)$ and $I(t)$ denote the susceptible and infected fractions.
The control $u(t)$ represents vaccination effort.
We assume normalized population size $S(t)+I(t)\le 1$.

For a fixed threshold $\varepsilon>0$, the eradication time is defined as
\[
T=\inf\{t\ge0:\ I(t)\le\varepsilon\}.
\]
The objective is to minimize $T$ over all admissible controls.

Throughout we work on the positive simplex and its closure,
\begin{equation}\label{eq:Omega_def}
\Omega:=\{(S,I)\in(0,1]^2:\ S+I\le 1\},
\qquad
\overline\Omega:=\{(S,I)\in\R^2:\ S\ge0,\ I\ge0,\ S+I\le 1\},
\end{equation}
and we impose the following assumptions.

\begin{assumption}\label{ass:main}
\begin{enumerate}[label=(A\arabic*)]
\item $\beta$ and $\gamma$ extend to $C^1$ functions on an open neighbourhood
of $\overline\Omega$, and $\beta,\gamma>0$ on $\overline\Omega$;
\item $S(0)>0$, $I(0)>0$, and $S(0)+I(0)\le1$;
\item $\beta(S,I)+I\partial_I\beta(S,I)>0$ on $\Omega$.
\end{enumerate}
\end{assumption}

\begin{rem}\label{rem:monotonicity}
Condition (A3) requires incidence to increase with prevalence at fixed $S$:
\[
\partial_I\bigl(\beta(S,I)SI\bigr)
=S\bigl(\beta+I\partial_I\beta\bigr)>0.
\]
It gives the positive coefficient in the infected costate equation
\eqref{eq:lambdaI_ODE} used in \Cref{lem:lambdaI_sign}.
The stronger condition $\partial_I\beta\ge0$ is sufficient but not necessary;
\Cref{rem:saturating_incidence} gives an example with
$\partial_I\beta<0$.  The numerical families in \Cref{sec:numerics} satisfy
$\partial_I\beta\ge0$, together with
\[
\partial_S\beta\le0
\qquad\text{and}\qquad
\partial_I\gamma\ge0,
\]
so their transmission coefficient decreases with $S$ and their recovery
rate increases with $I$.  These are properties of the selected
parametrizations, not general assumptions about behavioral responses or
healthcare capacity.  Neither condition is used in the analysis.
In particular, the quantity $F$ defined in
\eqref{eq:F_def}, the coefficient of $\lambda_I$ in
\eqref{eq:lambdaI_ODE}, is never
required to be sign definite.
\end{rem}

Forward invariance is then a consequence of the dynamics rather than an
assumption. Since
\[
\frac{\de}{\de t}\bigl(S+I\bigr)=-u\,S-\gamma(S,I)\,I\le 0,
\]
the quantity $S+I$ is nonincreasing, so $S(0)+I(0)\le1$ implies
$S(t)+I(t)\le1$ for all $t\ge0$. Together with the positivity of $S$ and $I$
on $[0,T)$ (\Cref{lem:positivity}), the domain $\Omega$ is forward invariant.

Finally, we denote by
\begin{equation*}
\cA:=\bigl\{u:[0,\infty)\to[0,1]\ \text{Lebesgue measurable}\bigr\}
\end{equation*}
the set of admissible controls.
\section{Pontryagin analysis and the single switch structure}
\label{sec:PMP_switch}

This section establishes the qualitative structure of time-optimal
vaccination policies for the exit time problem associated with
\eqref{eq:SIR}.  In the classical constant rate SIR model, the optimal
policy is known to be bang-bang with at most one switching time (from no
vaccination to full vaccination).  When the rates depend on state, the
adjoint system contains additional terms involving $\partial_I\beta$ and
$\partial_I\gamma$.  We identify a structural condition that controls the
costate signs and yields the single switch property through an integrating
factor for the switching function.

Throughout, we work on the forward invariant domain $\Omega$ defined in
\eqref{eq:Omega_def}, assume \Cref{ass:main}, and consider an optimal pair
$(S^*,I^*,u^*)$ with $I^*(0)>\varepsilon$ and minimal eradication time $T^*$.

\subsection{Hamiltonian and adjoint system}
\label{subsec:PMP_core}

Let $x=(S,I)$ and let $\lambda=(\lambda_S,\lambda_I)$ denote the adjoint
variables.  For a minimum time problem the maximum principle carries a
cost multiplier $\lambda_0\ge0$, and the Pontryagin Hamiltonian is
\begin{equation*}
H(x,u,\lambda_0,\lambda)
=
\lambda_0
+\lambda_S\bigl(-\beta(S,I)SI-uS\bigr)
+\lambda_I\bigl(\beta(S,I)SI-\gamma(S,I)I\bigr),
\qquad u\in[0,1].
\end{equation*}
Pontryagin's Maximum Principle for exit time problems
\citep{PontryaginEtAl1962,LeeMarkus1967,SchattlerLedzewicz2012} yields a
constant $\lambda_0\ge0$ and an absolutely continuous adjoint
$\lambda(\cdot)$, not both trivial, such that
\begin{equation}\label{eq:adjoint_full}
\dot\lambda(t)
=
-\nabla_x H\bigl(x^*(t),u^*(t),\lambda_0,\lambda(t)\bigr)^{\!\top}
\quad \text{for a.e. } t\in[0,T^*],
\end{equation}
together with
\begin{equation}\label{eq:H_minimized}
H\bigl(x^*(t),u^*(t),\lambda_0,\lambda(t)\bigr)
=
\min_{v\in[0,1]}H\bigl(x^*(t),v,\lambda_0,\lambda(t)\bigr)
=0
\quad\text{for a.e. }t.
\end{equation}
The minimized Hamiltonian is continuous in time, so its value is zero on all
of $[0,T^*]$.  The transversality condition at the hitting time is
\begin{equation}\label{eq:transversality}
\lambda_S(T^*)=0.
\end{equation}
The terminal constraint is $I(T^*)=\varepsilon$ with free terminal time, and
\eqref{eq:transversality} is the standard condition for this formulation:
the adjoint must annihilate the tangent space of the target boundary
$\{I=\varepsilon\}$, which is spanned by the direction of the $S$ coordinate.
Indeed, $S(T^*)>0$ and
$S(T^*)+I(T^*)<S(0)+I(0)\le1$, so no boundary of the invariant simplex is
active at the terminal point.

An extremal is called \emph{normal} when $\lambda_0>0$, in which case one may
rescale to $\lambda_0=1$ and recover the Hamiltonian
\begin{equation}\label{eq:Hamiltonian}
H(x,u,\lambda)
=
1
+\lambda_S\bigl(-\beta(S,I)SI-uS\bigr)
+\lambda_I\bigl(\beta(S,I)SI-\gamma(S,I)I\bigr).
\end{equation}
Normality is not automatic and must be established; we do so in
\Cref{prop:normality} below, where it is shown to be a consequence of the
structural condition of \Cref{ass:G_pos} rather than an extra hypothesis.
Until then no normalization of $\lambda_0$ is used.

\paragraph{Positivity before eradication.}
We first record that on any admissible trajectory, both compartments remain
positive until the target is reached.

\begin{lem}\label{lem:positivity}
Let $(S(t),I(t))$ solve \eqref{eq:SIR} with $S(0)>0$, $I(0)>0$ and
$u(\cdot)\in[0,1]$.  Then, at every finite time for which the solution exists,
\[
S(t)>0,\qquad I(t)>0.
\]
\end{lem}

\begin{proof}
The identities
\[
S(t)=S(t_0)\exp\!\left(
-\int_{t_0}^t [\beta(S(\tau),I(\tau))\,I(\tau)+u(\tau)]\,d\tau\right),\]
and
\[
I(t)=I(t_0)\exp\!\left(\int_{t_0}^t [\beta(S(\tau),I(\tau))\,S(\tau)-\gamma(S(\tau),I(\tau))]\,d\tau\right)
\]
hold for any $t\ge t_0$ as long as the solution exists, and imply the claim.
\end{proof}

\subsection{Switching function and the bang-bang rule}
\label{subsec:switching}

Since $H$ is affine in $u$, the optimal control is determined by the sign
of the switching function.  Define
\begin{equation*}
\psi(t)
:=
\frac{\partial H}{\partial u}\bigl(x^*(t),u^*(t),\lambda(t)\bigr)
=
-\lambda_S(t)\,S^*(t).
\end{equation*}
Because $u\mapsto H(\cdot,u,\cdot)$ is affine and $u\in[0,1]$, minimization
of $H$ over the control yields the endpoint rule
\begin{equation}\label{eq:bang_rule}
u^*(t)=
\begin{cases}
0,&\psi(t)>0,\\[2pt]
1,&\psi(t)<0,
\end{cases}
\qquad\text{for a.e. }t\in[0,T^*].
\end{equation}
Thus qualitative properties of $u^*$ reduce to the behavior of $\psi$.

\subsection{Adjoint equations in components}
\label{subsec:adjoint_components}

We record the componentwise adjoint equations, obtained directly from
$\dot\lambda=-\nabla_x H$ with $H$ as in \eqref{eq:Hamiltonian}.  Throughout,
quantities depending on state are evaluated along the optimal trajectory, and
subscripts on $\beta,\gamma$ denote partial derivatives.

\begin{lem}\label{lem:adjoint_components}
Along an optimal trajectory,
\begin{align}
\dot\lambda_S
&=\lambda_S\bigl(\beta I+\beta_S SI+u^*\bigr)
-\lambda_I\bigl(\beta I+\beta_S SI-\gamma_S I\bigr),\label{eq:lambdaS_ODE}\\
\dot\lambda_I
&=\lambda_S\,S\bigl(\beta+\beta_I I\bigr)
+\lambda_I\,F(S,I),\label{eq:lambdaI_ODE}
\end{align}
where
\begin{equation}\label{eq:F_def}
F(S,I):=\gamma+\gamma_I I-\beta S-\beta_I SI.
\end{equation}
\end{lem}

\begin{proof}
Immediate from $\dot\lambda_S=-\partial_S H$ and $\dot\lambda_I=-\partial_I H$,
using $\partial_S(\beta SI)=\beta I+\beta_S SI$, $\partial_I(\beta SI)=\beta S+\beta_I SI$,
and $\partial_I(\gamma I)=\gamma+\gamma_I I$.
\end{proof}

\begin{rem}
The derivatives $\beta_S,\gamma_S$ appear only in \eqref{eq:lambdaS_ODE},
hence directly in the switching dynamics (\Cref{lem:psi_dot} below); the
derivatives $\beta_I,\gamma_I$ appear only in \eqref{eq:lambdaI_ODE},
hence in the sign behaviour of $\lambda_I$.  In the constant rate case all
four vanish and one recovers the classical adjoint system.
\end{rem}

\subsection{Derivative of the switching function}
\label{subsec:psi_derivative}

The following identity is used to prove the single switch result.
Define the scalar field
\begin{equation}\label{eq:G_def}
G(S,I):=\beta(S,I)+S\,\partial_S\beta(S,I)-\partial_S\gamma(S,I)
=\partial_S\bigl(\beta(S,I)\,S-\gamma(S,I)\bigr),
\end{equation}
i.e.\ the sensitivity of the infected growth rate $\beta S-\gamma$ to the
susceptible fraction.

\begin{lem}\label{lem:psi_dot}
Along an optimal trajectory,
\begin{equation}\label{eq:psi_dot}
\dot\psi(t)
=
\beta_S(S,I)\,S I\,\psi(t)
+\lambda_I(t)\,S I\,G(S,I).
\end{equation}
In particular, at any time $t_0$ such that $\psi(t_0)=0$,
\begin{equation}\label{eq:psi_dot_at_zero}
\dot\psi(t_0)=\lambda_I(t_0)\,S^*(t_0)\,I^*(t_0)\,G\bigl(S^*(t_0),I^*(t_0)\bigr).
\end{equation}
\end{lem}

\begin{proof}
Differentiate $\psi=-\lambda_S S$ and use $\dot S=-\beta SI-u^*S$ together
with \eqref{eq:lambdaS_ODE}:
\[
\dot\psi=-\dot\lambda_S\,S-\lambda_S\,\dot S
=\lambda_S S\bigl[(\beta I+u^*)-(\beta I+\beta_S SI+u^*)\bigr]
+\lambda_I S\bigl(\beta I+\beta_S SI-\gamma_S I\bigr).
\]
The first bracket equals $-\beta_S SI$, and since $\lambda_S=-\psi/S$ the
first term is $\beta_S SI\,\psi$.  The second term equals
$\lambda_I SI(\beta+\beta_S S-\gamma_S)=\lambda_I SI\,G$.  This proves
\eqref{eq:psi_dot}; evaluating at a zero of $\psi$ gives
\eqref{eq:psi_dot_at_zero}.
\end{proof}

Equation \eqref{eq:psi_dot} is a linear ODE for $\psi$ whose inhomogeneous
term is $\lambda_I SI\,G$.  The
homogeneous coefficient $\beta_S SI$ merely rescales $\psi$ and never
changes its sign.  Equation \eqref{eq:psi_dot_at_zero} shows that the
direction in which $\psi$ crosses zero is governed by the sign of the
product $\lambda_I\,G$, hence by the sensitivity of the growth rate with
respect to $S$.  The derivatives $\partial_I\beta$ and $\partial_I\gamma$ do
not enter this crossing condition.

\subsection{A structural condition and exclusion of singular arcs}
\label{subsec:structural}

The constant rate case has $G=\beta>0$.  When the rates depend on state, we
impose the following condition.

\begin{assumption}[Structural sign condition]\label{ass:G_pos}
In addition to \Cref{ass:main}, assume that
\begin{equation}\label{eq:G_pos}
G(S,I)=\beta+S\,\partial_S\beta-\partial_S\gamma>0
\qquad\text{on }\Omega.
\end{equation}
\end{assumption}

\begin{rem}
Condition \eqref{eq:G_pos} means that the per capita growth rate
$\beta S-\gamma$ increases strictly with $S$ at fixed $I$.
If $\gamma_S=0$, this means that
$\beta S$ increases with $S$, even when $\beta$ itself decreases.
For general recovery rates the condition concerns net growth, not
transmission alone.
The condition may fail when the decrease of $\beta$ in $S$ is sufficiently
large to make $\beta S$ decreasing, or when $\gamma$ increases sufficiently
rapidly with $S$.
\end{rem}

The structural condition also implies a property of the exit that is
independent of the adjoint system and will be used to prove normality.  Write
\begin{equation}\label{eq:growth_rate}
\rho(S,I):=\beta(S,I)\,S-\gamma(S,I),
\qquad\text{so that}\qquad
\dot I=\rho(S,I)\,I
\ \text{ and }\ \partial_S\rho=G .
\end{equation}

\begin{lem}\label{lem:transversal_exit}
Assume \Cref{ass:main} and \Cref{ass:G_pos}.  Let $u(\cdot)\in\cA$, let
$(S,I)$ be the corresponding trajectory from an initial state in $\Omega$
with $I(0)>\varepsilon$, and let
$T=\inf\{t\ge0:\,I(t)\le\varepsilon\}<\infty$ be the first hitting time.
Then
\begin{equation*}
\dot I(T)=\rho\bigl(S(T),\varepsilon\bigr)\,\varepsilon<0 .
\end{equation*}
\end{lem}

\begin{proof}
Since $T$ is the first hitting time, $I(t)>\varepsilon$ for $t<T$ and
$I(T)=\varepsilon$.  The function $I$ is continuously differentiable because
its equation does not contain the measurable control explicitly.  Hence
$\dot I(T)\le0$ and $\rho(S(T),\varepsilon)\le0$.

Suppose that $\rho(S(T),\varepsilon)=0$.  Then $\dot I(T)=0$, and
\begin{equation}\label{eq:I_little_o_at_exit}
I(t)-\varepsilon=o(T-t)
\qquad\text{as }t\uparrow T.
\end{equation}
By \Cref{lem:positivity} and continuity, there are $c>0$ and $\eta>0$ such
that
\[
S(s)\beta(S(s),I(s))I(s)\ge c,
\qquad s\in[T-\eta,T].
\]
Since $u\ge0$, for $t\in[T-\eta,T]$,
\begin{equation}\label{eq:S_linear_at_exit}
S(t)-S(T)
=\int_t^T S(s)\bigl(\beta(S(s),I(s))I(s)+u(s)\bigr)\,\de s
\ge c(T-t).
\end{equation}
After decreasing $\eta$, continuity of $G$ gives $G\ge g>0$ near the exit
point.  Also, $|\partial_I\rho|\le L$ there for some $L>0$.  The mean value
theorem, \eqref{eq:I_little_o_at_exit}, and \eqref{eq:S_linear_at_exit} give
\begin{align*}
\rho(S(t),I(t))
&=\rho(S(t),I(t))-\rho(S(T),I(t))
  +\rho(S(T),I(t))-\rho(S(T),\varepsilon)\\
&\ge g\bigl(S(t)-S(T)\bigr)-L|I(t)-\varepsilon|\\
&>0
\end{align*}
for $t<T$ sufficiently close to $T$.  Thus $\dot I(t)>0$ on a left
neighbourhood of $T$.  This implies $I(t)<I(T)=\varepsilon$ there, which
contradicts the definition of $T$.  Therefore
$\rho(S(T),\varepsilon)<0$.
\end{proof}

\begin{rem}
The prevalence contribution is bounded by $L|I(t)-\varepsilon|=o(T-t)$,
whereas the susceptible contribution is bounded below by $gc(T-t)$.
Thus the proof requires no sign condition on $\partial_I\rho$ or
$\partial_I\gamma$; see \Cref{rem:monotonicity}.  Geometrically,
\Cref{lem:transversal_exit} says that a trajectory can only reach
$\{I=\varepsilon\}$ through the part of that boundary on which the infected
compartment is decreasing; it cannot arrive tangentially.
\end{rem}

We next exclude the abnormal case.

\begin{prop}\label{prop:normality}
Assume \Cref{ass:main} and \Cref{ass:G_pos}, and let
$(\lambda_0,\lambda(\cdot))$ be a nontrivial Pontryagin multiplier with
$\lambda_0\ge0$, satisfying \eqref{eq:adjoint_full},
\eqref{eq:H_minimized}, and the transversality condition
\eqref{eq:transversality} along an optimal pair with $I^*(0)>\varepsilon$.
Then $\lambda_0>0$.  Consequently one may normalize
$\lambda_0=1$, and \eqref{eq:Hamiltonian} applies.
\end{prop}

\begin{proof}
Suppose $\lambda_0=0$.  Evaluating the minimized Hamiltonian at $t=T^*$ and using
$\lambda_S(T^*)=0$ from \eqref{eq:transversality} leaves
$\lambda_I(T^*)\,\dot I(T^*)=0$.  By \Cref{lem:transversal_exit},
$\dot I(T^*)<0$, so $\lambda_I(T^*)=0$ and therefore $\lambda(T^*)=0$.  But
for fixed $u^*(\cdot)$ the adjoint equation \eqref{eq:adjoint_full} is a
linear homogeneous ODE for $\lambda$ with coefficients that are bounded and
measurable on $[0,T^*]$, so $\lambda(T^*)=0$ forces $\lambda\equiv0$ by
uniqueness for the backward Cauchy problem.  Then
$(\lambda_0,\lambda)\equiv(0,0)$, contradicting the nontriviality asserted by
the maximum principle.  Hence $\lambda_0>0$.
\end{proof}

\begin{rem}
\Cref{lem:transversal_exit} is a statement about the state trajectory alone:
its proof uses only the dynamics \eqref{eq:SIR}, the first hitting
definition of $T$, and $G>0$.  In particular it does not invoke the adjoint
system, the value of $\lambda_0$, or any normalization.  This independence
allows the lemma to be used to establish normality and thereby
to justify the terminal sign $\lambda_I(T^*)>0$ obtained in
\Cref{lem:lambdaI_sign} below.
\end{rem}

\begin{lem}\label{lem:no_singular}
Under \Cref{ass:main} and \Cref{ass:G_pos}, every zero of $\psi$ along an
optimal extremal with $I^*(0)>\varepsilon$ is transversal.  In particular,
$\psi$ has finitely many zeros on $[0,T^*]$.
\end{lem}

\begin{proof}
Let $\psi(t_0)=0$.  Since $S(t_0)>0$, we have $\lambda_S(t_0)=0$.  If also
$\lambda_I(t_0)=0$, uniqueness for the linear adjoint equation would give
$\lambda\equiv0$, in which case the minimized normal Hamiltonian equals $1$,
a contradiction.  Hence
$\lambda_I(t_0)\ne0$.  Equation \eqref{eq:psi_dot_at_zero} and
\Cref{ass:G_pos} now give $\dot\psi(t_0)\ne0$.  Here $\psi$ is continuously
differentiable because the right hand side of \eqref{eq:psi_dot} is
continuous.  Thus every zero is isolated.  The zero set is also compact, so
it is finite.
\end{proof}

\subsection{Sign structure of the infected costate and one switch optimality}
\label{subsec:one_switch_proof}

The following lemma uses the terminal conditions, \Cref{ass:main}, and
\Cref{ass:G_pos}.  No sign condition on $\partial_I\gamma$ is required.

\begin{lem}
\label{lem:lambdaI_sign}
Under \Cref{ass:main} and \Cref{ass:G_pos}, $\lambda_I(T^*)>0$ and
$\lambda_I$ has at most one
zero on $[0,T^*]$.  Consequently either $\lambda_I>0$ on all of $[0,T^*]$,
or there is $t_1\in[0,T^*)$ such that $\lambda_I(t_1)=0$,
$\lambda_I<0$ on $[0,t_1)$, and $\lambda_I>0$ on $(t_1,T^*]$.
\end{lem}

\begin{proof}
\emph{Terminal sign.} By \Cref{prop:normality} the extremal is normal, so we
may take $\lambda_0=1$ and use \eqref{eq:Hamiltonian}.  At $t=T^*$ we have
$\lambda_S(T^*)=0$ by \eqref{eq:transversality}, so the zero value of the
minimized Hamiltonian reduces to
$1+\lambda_I(T^*)\dot I(T^*)=0$.  By \Cref{lem:transversal_exit},
$\dot I(T^*)<0$, whence $\lambda_I(T^*)=-1/\dot I(T^*)>0$.

\emph{Behaviour at a zero.} Suppose $\lambda_I(t_1)=0$ for some $t_1$.  The
minimum of the Hamiltonian over $u\in[0,1]$ is continuous in time and equals
zero almost everywhere.  It therefore equals zero at $t_1$.  If
$\lambda_S(t_1)\le0$, then
\[
\min_{u\in[0,1]}
\{1-\lambda_S(t_1)S(t_1)(\beta I+u)\}\ge1,
\]
which is impossible.  Hence $\lambda_S(t_1)>0$.  Also,
\eqref{eq:lambdaI_ODE} gives
$\dot\lambda_I(t_1)=\lambda_S(t_1)\,S(t_1)\bigl(\beta+\beta_I I\bigr)(t_1)$,
and the factor $S(\beta+\beta_I I)>0$ by \Cref{ass:main}(A3)
and positivity of $S$.  Thus
$\dot\lambda_I(t_1)>0$.

\emph{At most one zero.}  The preceding calculation also shows that every
zero is isolated.  Since the zero set is closed in $[0,T^*]$, it is compact
and therefore finite.  If $t_a<t_b$ were two consecutive zeros, then
$\lambda_I>0$ just to the right of $t_a$ and, by absence of intermediate
zeros, $\lambda_I>0$ on $(t_a,t_b)$; but $\dot\lambda_I(t_b)>0$ forces
$\lambda_I<0$ just to the left of $t_b$, a contradiction.  Thus $\lambda_I$
has at most one zero, and together with $\lambda_I(T^*)>0$ and continuity
this yields the stated dichotomy.
\end{proof}

We now establish the one switch structure.  The function $\psi$ solves the
linear ODE \eqref{eq:psi_dot}, whose
inhomogeneous term changes sign at most once along the optimal trajectory,
which together with the transversality condition $\psi(T^*)=0$ forces at
most one interior sign change of $\psi$.

\begin{thm}\label{thm:single_switch}
Assume \Cref{ass:main} and \Cref{ass:G_pos}, with $I(0)>\varepsilon$. Then
any optimal control $u^*$ for the minimum eradication time problem
\eqref{eq:SIR} is bang-bang with at most one switch, and the switch is from
$u=0$ to $u=1$: there exists $\tau^*\in[0,T^*]$ such that
\[
u^{*}(t)=
\begin{cases}
0,&0\le t<\tau^{*},\\
1,&\tau^{*}<t\le T^*,
\end{cases}
\qquad\text{for a.e. }t\in[0,T^*].
\]
\end{thm}

\begin{proof}
By \Cref{lem:no_singular}, $u^*$ is bang-bang and determined by
$\operatorname{sign}\psi$ through \eqref{eq:bang_rule}, so it suffices to
show that $\psi$ has at most one interior sign change, from positive to
negative.

Let $\mu(t):=\exp\!\bigl(-\int_0^t\beta_S(S,I)\,SI\,\de s\bigr)>0$ and
$\Theta:=\mu\psi$, so that $\operatorname{sign}\Theta=\operatorname{sign}\psi$.
By \eqref{eq:psi_dot},
\[
\dot\Theta(t)=\mu(t)\,\lambda_I(t)\,S(t)I(t)\,G(S,I),
\]
hence $\operatorname{sign}\dot\Theta=\operatorname{sign}\lambda_I$, because
$\mu SI\,G>0$ by \Cref{ass:G_pos} and \Cref{lem:positivity}.  Transversality
gives $\Theta(T^*)=\mu(T^*)\psi(T^*)=0$.

We invoke the dichotomy of \Cref{lem:lambdaI_sign}.  If $\lambda_I>0$
throughout, then $\Theta$ is strictly increasing and $\Theta(T^*)=0$.
Therefore $\Theta<0$ on $[0,T^*)$, and $u^*=1$ almost everywhere.

In the second case, if $t_1=0$, the same conclusion follows because
$\lambda_I>0$ on $(0,T^*]$.  Suppose that $t_1>0$.  Then $\Theta$ is strictly
decreasing on $[0,t_1]$ and strictly increasing on $[t_1,T^*]$.  Since
$\Theta(T^*)=0$, we have $\Theta<0$ on $[t_1,T^*)$.  If
$\Theta(0)\le0$, then $\Theta<0$ on $(0,T^*)$ and $u^*=1$ almost
everywhere.  If $\Theta(0)>0$, then $\Theta$ has one zero
  $\tau^*\in(0,t_1)$.  In this case $\psi>0$ on $[0,\tau^*)$ and
$\psi<0$ on $(\tau^*,T^*)$.
In all cases $u^*$ has at most one switch, from $0$ to $1$.
\end{proof}

\begin{rem}
For constant rates, $G=\beta>0$ and $\beta+I\partial_I\beta=\beta>0$, so both
conditions hold automatically.  For rates that depend on state, $G>0$
controls the susceptible dependence, while \Cref{ass:main}(A3) gives the
infected costate sign structure.  The coefficient $F$ in
\eqref{eq:lambdaI_ODE} need not be sign definite.
\end{rem}

\begin{rem}\label{rem:saturating_incidence}
Consider
\[
\beta(S,I)=\frac{\beta_0}{(1+bS)(1+aI)},\qquad
\gamma(S,I)=\gamma_0,
\qquad a,b\ge0,\quad \beta_0,\gamma_0>0.
\]
These rates satisfy \Cref{ass:main}(A1), and direct differentiation gives
\[
\beta+I\partial_I\beta
=\frac{\beta_0}{(1+bS)(1+aI)^2}>0,\qquad
G=\frac{\beta_0}{(1+bS)^2(1+aI)}>0.
\]
Thus \Cref{thm:single_switch} applies for every admissible initial state
with $I(0)>\varepsilon$.  When $a>0$, the transmission coefficient decreases
with prevalence, since
\[
\partial_I\beta=-\frac{a\beta_0}{(1+bS)(1+aI)^2}<0,
\]
while incidence $\beta(S,I)SI$ increases and saturates as a function of $I$
at fixed $S$.  This family shows that the single switch result includes
decreasing transmission coefficients, not only the sufficient special case
$\partial_I\beta\ge0$.
\end{rem}

\begin{rem}
\label{rem:persistence}
The eradication time is defined as the first hitting time of the target set
$\mathcal G=\{I\le\varepsilon\}$, so our optimality results depend only on the
trajectory up to the hitting time.
In particular, the single switch structure in \Cref{thm:single_switch} does not
require any assumption preventing the trajectory from leaving $\mathcal G$
afterwards.
Nevertheless, for interpretability it is sometimes desirable to ensure that,
once the infection level falls below $\varepsilon$,
it remains below this threshold under any admissible control.
This can be guaranteed by an inward pointing condition on the boundary
$I=\varepsilon$, stated next.
\end{rem}

\begin{assumption}[No re-emergence after hitting]\label{ass:no_reemergence}
Fix $\varepsilon\in(0,1)$. Assume that
\begin{equation}\label{eq:inward_condition}
\beta(S,\varepsilon)\,S \;\le\; \gamma(S,\varepsilon)
\qquad \text{for all } S\in[0,1-\varepsilon].
\end{equation}
\end{assumption}

\begin{prop}
\label{prop:invariance_G}
Assume \Cref{ass:main,ass:no_reemergence} and let $u\in\cA$.
For the corresponding solution of \eqref{eq:SIR}, if
$I(t_0)\le\varepsilon$ for some $t_0\ge 0$, then
\[
I(t)\le \varepsilon \qquad \text{for all } t\ge t_0.
\]
In particular, once the trajectory hits the target set $\mathcal G=\{I\le\varepsilon\}$,
it cannot leave $\mathcal G$ afterwards (no re-emergence).
\end{prop}

\begin{proof}
Comparing the signs of $\dot I$ at a first boundary contact does not by
itself give a contradiction, because \eqref{eq:inward_condition} permits
$\dot I=0$ there.  We argue instead by a Gronwall estimate above the
boundary.

Write $\rho(S,I):=\beta(S,I)S-\gamma(S,I)$, so that $\dot I=\rho(S,I)I$, and
set $L:=\sup_{\overline\Omega}|\partial_I\rho|$, which is finite because
$\beta,\gamma\in C^1(\overline\Omega)$.  Suppose, for contradiction, that
$I(t_1)>\varepsilon$ for some $t_1>t_0$, and put
\[
\bar t:=\sup\{t\in[t_0,t_1]:\ I(t)\le\varepsilon\},
\]
so that $I(\bar t)=\varepsilon$ and $I(t)>\varepsilon$ on $(\bar t,t_1]$.

For $t\in(\bar t,t_1]$ we have $I(t)>\varepsilon$, hence
$S(t)\le 1-I(t)<1-\varepsilon$, so \eqref{eq:inward_condition} applies at
$(S(t),\varepsilon)$ and gives $\rho(S(t),\varepsilon)\le0$.  Therefore
\[
\rho\bigl(S(t),I(t)\bigr)
=\rho\bigl(S(t),\varepsilon\bigr)
+\int_{\varepsilon}^{I(t)}\partial_I\rho\bigl(S(t),\sigma\bigr)\,\de\sigma
\ \le\ L\,\bigl(I(t)-\varepsilon\bigr).
\]
Setting $m:=I-\varepsilon>0$ on $(\bar t,t_1]$ and using $I\le1$,
\[
\dot m=\rho(S,I)\,I\ \le\ L\,m\,I\ \le\ L\,m
\qquad\text{on }(\bar t,t_1],
\qquad m(\bar t)=0 .
\]
Gronwall's inequality yields $m(t)\le m(\bar t)e^{L(t-\bar t)}=0$ for
$t\in[\bar t,t_1]$, contradicting $m(t_1)>0$.  Hence $I(t)\le\varepsilon$ for
all $t\ge t_0$.
\end{proof}

\begin{rem}
If \eqref{eq:inward_condition} is strengthened to a strict inequality, the
argument is immediate: at a first contact $\bar t$ one would
need $\dot I(\bar t)\ge0$ whereas strictness forces $\dot I(\bar t)<0$.  Only
the borderline case $\beta(S,\varepsilon)S=\gamma(S,\varepsilon)$ requires the
comparison estimate above.
\end{rem}
\section{Eradication in finite time and explicit upper bounds}
\label{sec:bounds}

In addition to the single switch result in \Cref{sec:PMP_switch}, we derive
an explicit upper bound on the eradication time.  The bound provides a simulation
horizon and ensures that the exit time value function is finite under
the assumptions on the nonlinear rates.  We bound the optimal eradication
time by the eradication time of the admissible control $u\equiv1$.
The resulting estimate depends only on coarse extrema of
$\beta$ and $\gamma$ over the forward invariant domain.

\subsection{A uniform comparison bound}

Fix $\varepsilon\in(0,1)$.  Recalling \eqref{eq:Omega_def} and
\Cref{ass:main}(A1), $\beta$ and $\gamma$ are positive and bounded on
$\Omega$; define
\begin{equation*}
\begin{aligned}
\beta_{\max}&:=\sup_{(S,I)\in\Omega}\beta(S,I),
&\qquad 0<\beta_{\max}<\infty,\\
\gamma_{\min}&:=\inf_{(S,I)\in\Omega}\gamma(S,I),
&\qquad 0<\gamma_{\min}<\infty.
\end{aligned}
\end{equation*}

\begin{thm}
\label{thm:finite_time_strong}
Fix $(S_0,I_0)\in\Omega$ with $I_0>\varepsilon$, and let
$T^*(S_0,I_0,\varepsilon)$ denote the minimum time needed to reach
$I(t)\le\varepsilon$ under admissible controls $u(\cdot)\in[0,1]$.
Define
\begin{equation}\label{eq:tc_def}
t_c
:=
\max\!\Bigl\{0,\ \ln\!\Bigl(\frac{2\beta_{\max}S_0}{\gamma_{\min}}\Bigr)\Bigr\}.
\end{equation}
Then $T^*(S_0,I_0,\varepsilon)<\infty$ and
\begin{equation}\label{eq:Tstar_bound_strong}
T^*(S_0,I_0,\varepsilon)
\;\le\;
t_c
+
\frac{2}{\gamma_{\min}}
\left[
\beta_{\max}S_0\,t_c
+
\ln\!\Bigl(\frac{I_0}{\varepsilon}\Bigr)
\right].
\end{equation}
Moreover, in the \emph{subcritical regime} $\beta_{\max}S_0<\gamma_{\min}$,
one has the sharper estimate
\begin{equation}\label{eq:Tstar_subcritical_strong}
T^*(S_0,I_0,\varepsilon)
\;\le\;
\frac{1}{\gamma_{\min}-\beta_{\max}S_0}
\ln\!\Bigl(\frac{I_0}{\varepsilon}\Bigr).
\end{equation}
\end{thm}

\begin{proof}
Since $u\equiv 1$ is admissible, the optimal eradication time is bounded by
the eradication time under this strategy:
\begin{equation}\label{eq:opt_le_compare}
T^*(S_0,I_0,\varepsilon)\le T(1;S_0,I_0),
\end{equation}
where $T(1;S_0,I_0)$ denotes the first time $I(t)\le\varepsilon$ along the
solution to \eqref{eq:SIR} with $u\equiv 1$.

Under $u\equiv 1$, the susceptible equation satisfies
$\dot S=-\beta(S,I)SI-S\le -S$, hence
\begin{equation}\label{eq:S_decay}
S(t)\le S_0 e^{-t}\qquad\text{for all }t\ge 0.
\end{equation}
Introduce the threshold
\[
\delta:=\frac{\gamma_{\min}}{2\beta_{\max}}.
\]
By \eqref{eq:S_decay}, for all $t\ge t_c$ the solution satisfies
$S(t)\le \delta$ (this is exactly the choice of $t_c$ in \eqref{eq:tc_def}).

Next, for $t\ge t_c$ we estimate the infected dynamics:
\[
\dot I
=
\bigl(\beta(S,I)S-\gamma(S,I)\bigr)I
\le
(\beta_{\max}S(t)-\gamma_{\min})I
\le
(\beta_{\max}\delta-\gamma_{\min})I
=
-\frac{\gamma_{\min}}{2}\,I.
\]
Therefore,
\begin{equation}\label{eq:I_decay_after_tc}
I(t)\le I(t_c)\exp\!\Bigl(-\frac{\gamma_{\min}}{2}(t-t_c)\Bigr),
\qquad t\ge t_c.
\end{equation}

It remains to bound $I(t_c)$. For $0\le t\le t_c$ we have
\[
\dot I
=
(\beta(S,I)S-\gamma(S,I))I
\le
\beta(S,I)S\,I
\le
\beta_{\max}S(t)\,I
\le
\beta_{\max}S_0\,I,
\]
and Gr\"onwall's inequality yields
\begin{equation}\label{eq:I_tc_bound}
I(t_c)\le I_0\,\exp(\beta_{\max}S_0\,t_c).
\end{equation}
Combining \eqref{eq:I_decay_after_tc}--\eqref{eq:I_tc_bound}, it suffices to
choose $t\ge t_c$ such that
\[
\varepsilon
\ge
I_0\exp(\beta_{\max}S_0 t_c)\exp\!\Bigl(-\frac{\gamma_{\min}}{2}(t-t_c)\Bigr),
\]
which is equivalent to
\[
t
\ge
t_c+\frac{2}{\gamma_{\min}}
\left[\beta_{\max}S_0\,t_c+\ln\!\Bigl(\frac{I_0}{\varepsilon}\Bigr)\right].
\]
This gives an upper bound for $T(1;S_0,I_0)$, and together with
\eqref{eq:opt_le_compare} proves \eqref{eq:Tstar_bound_strong}.

Finally, in the subcritical regime $\beta_{\max}S_0<\gamma_{\min}$ we argue
directly, without invoking $t_c$: since $S(t)\le S_0$ for all $t\ge0$ (the
susceptible fraction is nonincreasing),
\[
\dot I\le (\beta_{\max}S_0-\gamma_{\min})I
=-(\gamma_{\min}-\beta_{\max}S_0)I,
\]
so $I(t)\le I_0 e^{-(\gamma_{\min}-\beta_{\max}S_0)t}$ and solving
$I(t)=\varepsilon$ yields \eqref{eq:Tstar_subcritical_strong}.
\end{proof}

\begin{rem}
The estimate \eqref{eq:Tstar_bound_strong} is \emph{uniform} in the sense
that it requires no further structure of $\beta(S,I)$ and $\gamma(S,I)$
beyond boundedness and positivity on $\Omega$.  The transient time $t_c$
is the time needed (under full vaccination) to drive $S(t)$ below the
threshold $\delta=\gamma_{\min}/(2\beta_{\max})$, after which the infected
population decays at a guaranteed exponential rate $\gamma_{\min}/2$.
In the numerical experiments, \eqref{eq:Tstar_bound_strong} provides a
simulation horizon and a reference bound.
\end{rem}

\subsection{Consequences for single switch policies}

The bound is conservative because it is derived from $u\equiv1$.  Under the
single switch structure of \Cref{thm:single_switch}, a smaller upper bound can
be obtained by optimizing over the switching time.

\begin{cor}\label{cor:two_policy_bound}
Let $T_1(S,I,\varepsilon)$ denote the eradication time under $u\equiv 1$.
Then for any $\tau\ge 0$,
\[
T^*(S_0,I_0,\varepsilon)
\le
\tau + T_1\!\bigl(\Phi^0_\tau(S_0,I_0),\varepsilon\bigr),
\]
where $\Phi^0$ is the flow of the uncontrolled dynamics ($u\equiv 0$).
In particular,
\[
T^*(S_0,I_0,\varepsilon)\le T_1(S_0,I_0,\varepsilon).
\]
\end{cor}

\begin{proof}
Fix $\tau\ge 0$ and consider the admissible switching control
$u_\tau(t)=0$ for $t<\tau$ and $u_\tau(t)=1$ for $t\ge \tau$.
The trajectory coincides with $\Phi^0_t(S_0,I_0)$ on $[0,\tau]$ and then
evolves under $u\equiv 1$ from the state $\Phi^0_\tau(S_0,I_0)$, unless it
has already reached the target.  In the first case its eradication time is
$\tau + T_1(\Phi^0_\tau(S_0,I_0),\varepsilon)$.  In the second case it is at
most $\tau$.  The stated inequality follows in both cases.
\end{proof}

\begin{rem}
\Cref{cor:two_policy_bound} anticipates the dynamic programming
representation developed in \Cref{sec:DPP_switch}: once the post-switch
value under $u\equiv 1$ is understood, bounding (and computing) the
optimal eradication time reduces to optimization over the switching time.
\end{rem}

\subsection{Explicit bounds for representative nonlinear scenarios}
\label{subsec:scenario_extrema}

To use \Cref{thm:finite_time_strong} in numerical experiments, one needs
conservative estimates of $\beta_{\max}$ and $\gamma_{\min}$ on $\Omega$.
For many parametrizations depending on state, simple explicit bounds are
available from monotonicity and the constraints $S\in(0,1]$, $I\in(0,1]$.

\paragraph{Scenario (A): saturating contact.}
For
\[
\beta(S,I)=\beta_0\frac{1+aI}{1+bS},
\qquad
\gamma(S,I)=\gamma_0+cI,
\]
with $\beta_0,\gamma_0>0$ and $a,b,c\ge0$, one has
\begin{equation}\label{eq:scenarioA_extrema}
\beta_{\max}=\beta_0(1+a),\qquad
\gamma_{\min}=\gamma_0,
\end{equation}
since $1+bS\ge 1$ and $0\le I\le 1$ on $\overline\Omega$.

\paragraph{Scenario (B): exponential transmission.}
For
\[
\beta(S,I)=\beta_0\exp(\alpha I-\sigma S),
\qquad
\gamma(S,I)=\gamma_0+cI,
\]
with $\beta_0,\gamma_0>0$ and $\alpha,\sigma,c\ge0$, we obtain
\begin{equation*}
\beta_{\max}=\beta_0 e^{\alpha},\qquad
\gamma_{\min}=\gamma_0,
\end{equation*}
using $I\le 1$ and $-\sigma S\le 0$ on $\Omega$.

\paragraph{Scenario (C): power law heterogeneity.}
For
\[
\beta(S,I)=\beta_0(1+k_I I)^{m}(1-k_S S)^{n},
\qquad
\gamma(S,I)=\gamma_0+cI,
\]
with $\beta_0,\gamma_0>0$, $k_I,c\ge0$, $0\le k_S<1$, and $m,n\ge0$, we
have
\begin{equation}\label{eq:scenarioC_extrema}
\beta_{\max}=\beta_0(1+k_I)^{m},\qquad
\gamma_{\min}=\gamma_0.
\end{equation}

\begin{rem}
Substituting \eqref{eq:scenarioA_extrema}--\eqref{eq:scenarioC_extrema}
into \Cref{thm:finite_time_strong} yields fully explicit simulation
horizons in terms of scenario parameters and initial data.  The extrema
above are exact on $\overline\Omega$: the maximum of $\beta$ is attained at
$(0,1)$ and the minimum of $\gamma$ on $I=0$.  The time bound remains
conservative because it uses global rate extrema and a uniform comparison
estimate, rather than the rates encountered along a particular trajectory.
Optimizing these same extrema over the triangle would not improve them.
\end{rem}
\section{Dynamic programming and the Hamilton--Jacobi equation}
\label{sec:DPP_switch}

This section establishes existence of an optimal control and derives the
switching time representation by dynamic programming.  It also relates
the problem to a Hamilton--Jacobi equation and identifies the boundary
behavior relevant to that interpretation.

\subsection{Exit time value function and dynamic programming principle}
\label{subsec:DPP}

For an admissible control $u(\cdot)\in\cA$, let
\[
T^u(S,I):=\inf\{t\ge0:\ I^u(t)\le\varepsilon\}
\]
denote the eradication time.
We define the exit time value function
\begin{equation}\label{eq:V_def}
V(S,I):=\inf_{u(\cdot)\in\cA} T^u(S,I),
\qquad (S,I)\in\Omega.
\end{equation}
By definition, $V(S,I)=0$ for all $(S,I)\in\mathcal G:=\{I\le\varepsilon\}$.

For $h\ge0$, set $\theta^u:=\min\{h,T^u(S,I)\}$.  The stopped dynamic
programming principle is
\begin{equation}\label{eq:DPP_V}
V(S,I)
=
\inf_{u(\cdot)\in\cA}
\Bigl\{
\theta^u + V\bigl(S^u(\theta^u),I^u(\theta^u)\bigr)
\Bigr\},
\qquad h\ge0.
\end{equation}
This form remains valid when a trajectory reaches the target before $h$.
The principle is standard \citep[Ch.~IV]{BardiCapuzzoDolcetta1997}.  We use
only the following consequence, which does not require regularity of $V$.

\begin{prop}
\label{prop:optimality_principle}
Assume \Cref{ass:main}.  Let $z\in\Omega$ with $I>\varepsilon$, let $u^*$ be
an optimal control for $z$ and $x^*(\cdot)$ the corresponding trajectory.
Then
\begin{equation*}
V\bigl(x^*(t)\bigr)=V(z)-t
\qquad\text{for every } t\in[0,V(z)].
\end{equation*}
\end{prop}

\begin{proof}
Fix $t\in[0,V(z)]$ and recall $V(z)=T^{u^*}(z)$.

\emph{Upper bound.}  The shifted control $s\mapsto u^*(t+s)$ is admissible
from $x^*(t)$ and steers it into $\mathcal G$ in time $V(z)-t$, because
$x^*$ first reaches $\mathcal G$ at time $V(z)$.  Hence
$V(x^*(t))\le V(z)-t$.

\emph{Lower bound.}  The case $t=V(z)$ follows from
$x^*(V(z))\in\mathcal G$.  Let $t<V(z)$.  By
\Cref{prop:existence_opt}(E3) there is a control
$u'$ optimal from $x^*(t)$.  The concatenation equal to $u^*$ on $[0,t)$ and
to $u'(\cdot-t)$ on $[t,\infty)$ is admissible from $z$.  Along it the
trajectory coincides with $x^*$ on $[0,t]$, which does not meet
$\mathcal G$ before time $V(z)\ge t$, and thereafter reaches $\mathcal G$
at time $t+V(x^*(t))$.  Therefore $V(z)\le t+V(x^*(t))$.
\end{proof}

\subsection{Formal Hamilton--Jacobi equation}
\label{subsec:HJB_formal}

Let
\[
f(S,I,u)
=
\bigl(-\beta(S,I)SI-uS,\ \beta(S,I)SI-\gamma(S,I)I\bigr)
\]
denote the controlled vector field.
Formally expanding \eqref{eq:DPP_V} for small $\tau>0$ yields the
Hamilton--Jacobi equation
\begin{equation}\label{eq:HJB_formal}
\inf_{u\in[0,1]}
\Bigl\{
1+\nabla V(S,I)\cdot f(S,I,u)
\Bigr\}
=0,
\qquad (S,I)\in\operatorname{int}(\Omega)\setminus\mathcal G,
\end{equation}
with boundary condition
\[
V(S,I)=0 \qquad \text{for } (S,I)\in\mathcal G.
\]

Since the Hamiltonian is affine in $u$, the minimization can be carried out
explicitly, leading to
\begin{equation}\label{eq:HJB_explicit}
\beta(S,I)SI\,V_S
+(\gamma(S,I)-\beta(S,I)S)I\,V_I
+S\,(V_S)^+
=
1,
\qquad (S,I)\in\operatorname{int}(\Omega)\setminus\mathcal G,
\end{equation}
where $(x)^+=\max\{x,0\}$.

\subsection{Viscosity solution interpretation}
\label{subsec:viscosity}

We briefly recall the viscosity solution framework associated with
\eqref{eq:HJB_explicit}.
Define the Hamiltonian
\[
\mathcal H\bigl((S,I),(p_S,p_I)\bigr)
:=
\sup_{u\in[0,1]}
\Bigl\{
-1+p_S(\beta SI+uS)+p_I(\gamma I-\beta SI)
\Bigr\}.
\]
This is the negative of the minimized expression in
\eqref{eq:HJB_formal}.  The sign is relevant for viscosity inequalities:
at a local maximum of $V-\varphi$, dynamic programming gives
$1+\nabla\varphi\cdot f(\cdot,u)\ge0$ for every constant $u$, hence
$\mathcal H(\cdot,\nabla\varphi)\le0$.

\begin{defn}
A continuous function $V:\Omega\to\R$ is a viscosity subsolution
(respectively, supersolution) of \eqref{eq:HJB_explicit} if, for every
$\varphi\in C^1(\Omega)$ such that $V-\varphi$ attains a local maximum
(respectively, minimum) at
$(S,I)\in\operatorname{int}(\Omega)\setminus\mathcal G$, one has
\[
\mathcal H\bigl((S,I),\nabla\varphi(S,I)\bigr)\le 0
\quad
(\text{respectively, } \ge 0).
\]
\end{defn}

Continuity of the controlled flow and finiteness of the exit time do not imply
that $V$ is a continuous viscosity solution of \eqref{eq:HJB_explicit} with
boundary datum $V=0$ on $\mathcal G$.  Part of the target boundary is an
\emph{outflow} boundary, from which trajectories move away for every
admissible control.  This issue can occur for constant rates as well: it
concerns the convention $T=0$ on the entire target, not state dependence
alone.

\begin{prop}
\label{prop:V_discontinuous}
Assume \Cref{ass:main} and suppose there exists
$\bar S\in(0,1-\varepsilon)$ with
\begin{equation}\label{eq:outflow}
\beta(\bar S,\varepsilon)\,\bar S-\gamma(\bar S,\varepsilon)>0 .
\end{equation}
Then $V$ is discontinuous at $\bar z:=(\bar S,\varepsilon)$.  Indeed
$V(\bar z)=0$, while there exist $t_0>0$ and $\varrho>0$ such that
\[
V(\bar S,\varepsilon+\delta)\ \ge\ t_0
\qquad\text{for all }\delta\in(0,\varrho).
\]
\end{prop}

\begin{rem}
The strict inequality $\bar S<1-\varepsilon$ ensures that the perturbed states
$(\bar S,\varepsilon+\delta)$ lie in $\Omega$ for small $\delta$.  At the
corner $\bar S=1-\varepsilon$ one perturbs along the constraint instead,
replacing $(\bar S,\varepsilon+\delta)$ by $(\bar S-\delta,\varepsilon+\delta)$;
the same argument applies, since it uses only that the perturbed
state lies in $B_{r/2}(\bar z)\cap\Omega$ with prevalence exceeding
$\varepsilon$.
\end{rem}

\begin{proof}
Since $\beta,\gamma\in C^1(\overline\Omega)$, by \eqref{eq:outflow} there are
$\eta>0$ and $r>0$ such that $\beta S-\gamma\ge\eta$ on
$B_r(\bar z)\cap\overline\Omega$.  The vector field $f$ is bounded on
$\overline\Omega\times[0,1]$, say by a constant $M>0$, so every admissible trajectory
started in $B_{r/2}(\bar z)$ remains in $B_r(\bar z)$ for
$t\le t_0:=r/(4M)$.  On that time interval, in the variable $L=\log I$,
\[
\dot L=\beta S-\gamma\ \ge\ \eta\ >\ 0 ,
\]
so $I(t)\ge I(0)e^{\eta t}\ge I(0)$.  Set
\[
\varrho:=\min\left\{\frac r2,\frac{1-\varepsilon-\bar S}{2}\right\}>0.
\]
Starting from $(\bar S,\varepsilon+\delta)$ with $0<\delta<\varrho$ we then
have
$I(t)\ge\varepsilon+\delta>\varepsilon$ for all $t\le t_0$ and all
admissible controls, whence $T^u(\bar S,\varepsilon+\delta)>t_0$ for every
$u$, and thus $V(\bar S,\varepsilon+\delta)\ge t_0$.  As $V(\bar z)=0$ by
definition, $V$ is discontinuous at $\bar z$.
\end{proof}

The condition $G>0$ also determines the sign changes of $\rho$ on the target
boundary.

\begin{cor}
\label{cor:boundary_arcs}
Assume \Cref{ass:main} and \Cref{ass:G_pos}, and consider the restriction of
the growth rate \eqref{eq:growth_rate} to the target boundary,
\begin{equation}\label{eq:boundary_growth}
S\ \longmapsto\ \rho(S,\varepsilon)=\beta(S,\varepsilon)\,S-\gamma(S,\varepsilon),
\qquad S\in(0,1-\varepsilon].
\end{equation}
Then $\partial_S\rho(\cdot,\varepsilon)=G(\cdot,\varepsilon)>0$, so
\eqref{eq:boundary_growth} is strictly increasing, and
$\rho(0^+,\varepsilon)=-\gamma(0,\varepsilon)<0$.  Consequently exactly one of
the following holds.
\begin{enumerate}[label=(\roman*),leftmargin=20pt]
\item $\rho(1-\varepsilon,\varepsilon)\le0$.  Then
$\rho(\cdot,\varepsilon)\le0$ on the whole boundary.  If the inequality at
$S=1-\varepsilon$ is strict, the boundary is strictly inward pointing.
\item $\rho(1-\varepsilon,\varepsilon)>0$.  Then $\rho(\cdot,\varepsilon)$ has
a unique zero $S_c\in(0,1-\varepsilon)$, the boundary splits into an
\emph{inflow arc} $\{S<S_c\}$ and an \emph{outflow arc} $\{S>S_c\}$ separated
by the characteristic point $S_c$.  The value function is discontinuous at
each point of the outflow arc by \Cref{prop:V_discontinuous}.
\end{enumerate}
\end{cor}

\begin{proof}
Differentiating \eqref{eq:boundary_growth} gives
$\partial_S\rho=\beta+S\partial_S\beta-\partial_S\gamma=G(S,\varepsilon)$,
positive by \Cref{ass:G_pos}; and $\rho(0^+,\varepsilon)<0$ because
$\gamma>0$ on $\overline\Omega$.  A strictly increasing continuous function
with a negative left limit has at most one zero, which gives the dichotomy.
The last assertion in case~(ii) follows from
\Cref{prop:V_discontinuous} and the corner modification stated after that
proposition.
\end{proof}

\begin{rem}
\label{rem:viscosity_scope}
Our convention assigns zero exit time to every initial state in the target.
In contrast, \citet{HyndEradication2021} use a positive-time threshold hitting
convention, so their continuous boundary values on the outflow arc need not
vanish.  The discontinuity above does not contradict their result.
Under the strict inward condition
\[
\beta(S,\varepsilon)S<\gamma(S,\varepsilon)
\qquad\text{for all }S\in[0,1-\varepsilon],
\]
the standard exit time theory gives continuity of $V$ and the viscosity
property \citep[Ch.~IV]{BardiCapuzzoDolcetta1997}.  With an outflow arc,
\Cref{prop:V_discontinuous} prevents a continuous formulation with zero data
on the entire target.  We do not establish a comparison principle,
uniqueness, or semiconcavity for $V$.  The switching theorem uses the maximum
principle, and the switching time representation uses dynamic programming
and the regularity of the fixed-control exit time $W$
(\Cref{lem:W_regularity}).  The feedback formula
$u^*=\mathbf 1_{\{V_S>0\}}$ is interpreted at differentiability points of $V$.
\end{rem}

\subsection{Value after switching and reduction to one variable}
\label{subsec:post_switch}

By \Cref{thm:single_switch}, any optimal control switches at most once from
$u=0$ to $u=1$.
Accordingly, define the post-switch value function
\begin{equation}\label{eq:W_def}
W(S,I)
:=
\inf\bigl\{
t\ge0:\ \pi_I(\Phi^1_t(S,I))\le\varepsilon
\bigr\},
\end{equation}
where $\Phi^1_t$ denotes the flow associated with $u\equiv 1$.  In
particular $W\equiv0$ on the target set $\mathcal G=\{I\le\varepsilon\}$, the
infimum being attained at $t=0$ there, and $W>0$ on
$\Omega\setminus\mathcal G$.

The function $W$ satisfies the exit time identity
\[
W(S,I)=h+W(\Phi^1_h(S,I)),
\qquad 0\le h\le W(S,I).
\]

The function $W$ is an exit time for the fixed smooth vector field
$f(\cdot,1)$.  Transversal exit from \Cref{lem:transversal_exit} allows us to
differentiate $W$ by the implicit function theorem.

\begin{lem}
\label{lem:W_regularity}
Assume \Cref{ass:main} and \Cref{ass:G_pos}.  Then $W$ admits a continuously
differentiable local extension about each point of
$\Omega\setminus\mathcal G$, and it solves
\begin{equation}\label{eq:W_transport}
1+\nabla W\cdot f(\cdot,1)=0
\end{equation}
there in the classical sense, and for every $z_0\in\Omega$ with
$I_0>\varepsilon$, the map
$\tau\mapsto\phi(\tau)=\tau+W(\Phi^0_\tau(z_0))$ is $C^1$ on every interval
on which $\Phi^0_\tau(z_0)\notin\mathcal G$, with
\begin{equation}\label{eq:phi_prime_rigorous}
\phi'(\tau)=1+\nabla W\bigl(\Phi^0_\tau(z_0)\bigr)\cdot f_0\bigl(\Phi^0_\tau(z_0)\bigr).
\end{equation}
\end{lem}

\begin{proof}
Fix $z_0\in\Omega\setminus\mathcal G$ and write
$(S_e,\varepsilon):=\Phi^1_{W(z_0)}(z_0)$ for the exit point under
$u\equiv1$, the exit time being finite by \Cref{thm:finite_time_strong}.
Applying \Cref{lem:transversal_exit} to the constant control $u\equiv1$
gives
\begin{equation}\label{eq:W_transversal}
\rho\bigl(S_e,\varepsilon\bigr)
=\beta(S_e,\varepsilon)S_e-\gamma(S_e,\varepsilon)<0 .
\end{equation}

Set $\Gamma(t,z):=\pi_I\bigl(\Phi^1_t(z)\bigr)-\varepsilon$.  Since
$f(\cdot,1)$ is $C^1$ on a neighbourhood of $\overline\Omega$ by
\Cref{ass:main}(A1), the flow
$\Phi^1$ is $C^1$ jointly in $(t,z)$, hence so is $\Gamma$.  By construction
$\Gamma(W(z_0),z_0)=0$, and
\[
\partial_t \Gamma\bigl(W(z_0),z_0\bigr)
=
\rho\bigl(S_e,\varepsilon\bigr)\,\varepsilon
<0
\]
by \eqref{eq:W_transversal}.  The implicit function theorem then provides a
neighbourhood of $z_0$, a neighbourhood of $W(z_0)$, and a unique $C^1$ map
$z\mapsto\mathcal T(z)$ with
$\mathcal T(z_0)=W(z_0)$ and $\Gamma(\mathcal T(z),z)=0$.  Choose
$\delta\in(0,W(z_0))$ small enough that this is the only zero in
$(W(z_0)-\delta,W(z_0)+\delta)$.  The reference trajectory has positive
distance from $\{I=\varepsilon\}$ on $[0,W(z_0)-\delta]$.  Continuous
dependence excludes a crossing on this interval for nearby initial states,
and \eqref{eq:W_transversal} makes the zero near $W(z_0)$ a downward
crossing.  Thus $\mathcal T=W$ near $z_0$, and
$W\in C^1$ there, with
\[
\nabla W(z_0)
=
-\frac{\nabla_z \Gamma\bigl(W(z_0),z_0\bigr)}
       {\partial_t \Gamma\bigl(W(z_0),z_0\bigr)} .
\]
Differentiating the identity $W(\Phi^1_h(z))=W(z)-h$ at $h=0$ gives
\eqref{eq:W_transport}, and the chain rule applied to
$\tau\mapsto W(\Phi^0_\tau(z_0))$, using
$\frac{\de}{\de\tau}\Phi^0_\tau(z_0)=f_0(\Phi^0_\tau(z_0))$, gives
\eqref{eq:phi_prime_rigorous}.
\end{proof}

All six parameter sets in \Cref{tab:scenarios} have an outflow arc, where
$V$ is discontinuous by \Cref{prop:V_discontinuous}.  Since $W\ge V$ and
both vanish on the target, $W$ is also discontinuous there.  It is
$C^1$ on $\Omega\cap\{I>\varepsilon\}$ by \Cref{lem:W_regularity}.  Over more than two
thousand initial states across these sets, the largest observed value of
$\rho(S_e,\varepsilon)$ at exit was $-1.5\times10^{-2}$.
Method~3 in \Cref{subsec:compute_tau} uses \eqref{eq:phi_prime_rigorous}.

\subsection{Switching time representation}
\label{subsec:switch_representation}

Combining the single switch structure with the dynamic programming
principle yields the following representation.

Before stating it, we establish existence of state trajectories and an
optimal control, so that no
hypothesis beyond \Cref{ass:main} is needed at this point.

\begin{prop}
\label{prop:existence_opt}
Assume \Cref{ass:main}. Then the following hold.
\begin{enumerate}[label=(E\arabic*)]
\item For every $u(\cdot)\in\cA$ and every initial state
$(S,I)\in\overline\Omega$, the controlled system \eqref{eq:SIR} admits a
unique Carath\'eodory solution on $[0,\infty)$, and this solution remains in
$\overline\Omega$ for all $t\ge0$.
\item The target is reachable in finite time from every $(S,I)\in\Omega$
with $I>\varepsilon$; in particular $V(S,I)<\infty$ on $\Omega$.
\item For every $(S,I)\in\Omega$ the infimum in \eqref{eq:V_def} is attained:
there exists $u^*\in\cA$ with $T^{u^*}(S,I)=V(S,I)$.
\end{enumerate}
\end{prop}

\begin{proof}
\emph{(E1).}  By \Cref{ass:main}(A1) the map $x\mapsto f(x,u)$ is $C^1$ on
the compact set $\overline\Omega$ with derivatives bounded uniformly in
$u\in[0,1]$, hence Lipschitz in $x$ with a constant $L$ independent of $u$;
moreover $f$ is affine in $u$ and bounded on $\overline\Omega\times[0,1]$.
Consequently, for measurable $u(\cdot)$ the map $(t,x)\mapsto f(x,u(t))$ is
measurable in $t$, Lipschitz in $x$ uniformly in $t$, and dominated by an
integrable function, so the Carath\'eodory existence and uniqueness theorem
applies and yields a unique local absolutely continuous solution.
Invariance of $\overline\Omega$ follows from the structure of
\eqref{eq:SIR}: on $\{S=0\}$ one has $\dot S=0$ and on $\{I=0\}$ one has
$\dot I=0$, so neither variable can become negative, while
$\frac{\de}{\de t}(S+I)=-uS-\gamma I\le0$ keeps $S+I\le1$.  The solution
therefore remains in the compact set $\overline\Omega$, cannot blow up, and
extends to all of $[0,\infty)$.

\emph{(E2).}  This is \Cref{thm:finite_time_strong}: the admissible control
$u\equiv1$ drives the trajectory into $\mathcal G$ in finite time, and
$V\le T(1;\cdot)$.

\emph{(E3).}  We verify the hypotheses of Filippov's existence theorem for
time-optimal control problems \citep{Filippov1962}; see also
\citet[Ch.~9]{Cesari1983}.  The control set $[0,1]$ is compact; $f$ is
continuous on $\overline\Omega\times[0,1]$ and Lipschitz in $x$ uniformly in
$u$ by (E1); every trajectory issuing from $\overline\Omega$ remains in the
compact set $\overline\Omega$; the target
$\mathcal G=\{(S,I)\in\overline\Omega:\ I\le\varepsilon\}$ is closed; and
$\mathcal G$ is reachable in finite time by (E2).  It remains to check the
structural condition, namely convexity of the velocity set.  Since $f$ is
affine in $u$,
\[
f\bigl(x,[0,1]\bigr)
=
\Bigl\{
\bigl(-\beta SI-uS,\ \beta SI-\gamma I\bigr):\ u\in[0,1]
\Bigr\}
\]
is, for each fixed $x=(S,I)\in\overline\Omega$, the straight segment joining
$\bigl(-\beta SI,\ \beta SI-\gamma I\bigr)$ to
$\bigl(-\beta SI-S,\ \beta SI-\gamma I\bigr)$.  It is therefore convex and
compact for every $x$.  Filippov's theorem now yields an admissible control
$u^*$ attaining the minimal time.
\end{proof}

\begin{rem}
The velocity set is convex because vaccination enters \eqref{eq:SIR}
affinely.  No relaxation of the
control set is required, and the value function of the relaxed problem
coincides with $V$.  Note also that \Cref{prop:existence_opt} is logically
independent of the Pontryagin analysis of \Cref{sec:PMP_switch}: the latter
presupposes an optimal pair, whose existence follows from
\Cref{prop:existence_opt}(E3).
\end{rem}

\begin{thm}
\label{thm:switch_rep}
Assume \Cref{ass:main} and \Cref{ass:G_pos}. For any $(S,I)\in\Omega$,
\begin{equation}\label{eq:V_switch_rep}
V(S,I)
=
\inf_{\tau\ge0}
\Bigl\{
\tau + W\bigl(\Phi^0_\tau(S,I)\bigr)
\Bigr\},
\end{equation}
where $\Phi^0$ is the flow associated with $u\equiv 0$ and
$W$ is the post-switch value under $u\equiv 1$ defined in \eqref{eq:W_def}.
\end{thm}

\begin{proof}
Fix an initial state $z:=(S,I)\in\Omega$.  If $z\in\mathcal G$, both sides of
\eqref{eq:V_switch_rep} are zero by taking $\tau=0$.  We therefore assume
$I>\varepsilon$.

For any $\tau\ge0$, define the single switch control
\[
u_\tau(t)=
\begin{cases}
0,& 0\le t<\tau,\\
1,& t\ge \tau.
\end{cases}
\]
Let $z_\tau:=\Phi^0_\tau(z)$ be the state reached at time $\tau$ under
$u\equiv 0$.
By admissibility of $u_\tau$ and the definition of $V$,
\begin{equation}\label{eq:V_le_Tutau}
V(z)\le T^{u_\tau}(z),
\end{equation}
where $T^{u_\tau}(z)$ is the eradication time under $u_\tau$.

We claim that
\begin{equation}\label{eq:Tu_tau_formula}
T^{u_\tau}(z)\le \tau + W(z_\tau).
\end{equation}
Indeed, if the trajectory under $u_\tau$ hits the target set $\mathcal G$
before time $\tau$, then $T^{u_\tau}(z)\le \tau$ and since $W\ge 0$ we obtain
$T^{u_\tau}(z)\le \tau \le \tau+W(z_\tau)$.
Otherwise, the trajectory has not reached $\mathcal G$ before $\tau$, and
the eradication time decomposes as
\[
T^{u_\tau}(z)=\tau + T^{u\equiv 1}(z_\tau)=\tau+W(z_\tau),
\]
by the definition of $W$ (eradication time under full vaccination from
the current state) and the semigroup property of the flow.
This proves \eqref{eq:Tu_tau_formula}.

Combining \eqref{eq:V_le_Tutau}--\eqref{eq:Tu_tau_formula} yields
\[
V(z)\le \tau + W(\Phi^0_\tau(z))\qquad \text{for all }\tau\ge0.
\]
Taking the infimum over $\tau$ gives
\begin{equation}\label{eq:V_le_inf}
V(z)\le \inf_{\tau\ge0}\{\tau+W(\Phi^0_\tau(z))\}.
\end{equation}

By \Cref{prop:existence_opt}(E3), there exists an optimal control $u^*$
attaining the value $V(z)=T^{u^*}(z)$.
By \Cref{thm:single_switch}, this optimal control is bang-bang with at most
one switch from $0$ to $1$.
Hence there exists $\tau^*\in[0,T^{u^*}(z)]$ such that
\[
u^*(t)=
\begin{cases}
0,& 0\le t<\tau^*,\\
1,& \tau^*\le t\le T^{u^*}(z),
\end{cases}
\qquad\text{a.e.}
\]
Let $z^*_{\tau^*}:=\Phi^0_{\tau^*}(z)$ denote the state reached at time
$\tau^*$ along the pre-switch dynamics.
Then, by the same decomposition argument as above (here it is exact because
$\tau^*\le T^{u^*}(z)$ by construction),
\[
V(z)=T^{u^*}(z)=\tau^*+W(z^*_{\tau^*})=\tau^*+W(\Phi^0_{\tau^*}(z)).
\]
Therefore,
\begin{equation}\label{eq:V_ge_inf}
V(z)\ge \inf_{\tau\ge0}\{\tau+W(\Phi^0_\tau(z))\}.
\end{equation}

Combining \eqref{eq:V_le_inf} and \eqref{eq:V_ge_inf} yields
\eqref{eq:V_switch_rep}.
\end{proof}

\begin{cor}
\label{cor:tau_star}
Assume \Cref{ass:main} and \Cref{ass:G_pos}, fix $z=(S,I)\in\Omega$ with
$I>\varepsilon$, and let
\[
\mathcal T(z):=
\bigl\{\tau\in[0,V(z)]:T^{u_\tau}(z)=V(z)\bigr\},
\qquad
\tau^*(z):=\min\mathcal T(z),
\]
where $u_\tau=0$ before $\tau$ and $u_\tau=1$ after $\tau$.  Then
\begin{equation}\label{eq:value_matching}
\tau^*(z)
=
\min\Bigl\{
t\ge0:\
V\bigl(\Phi^0_t(z)\bigr)=W\bigl(\Phi^0_t(z)\bigr)
\Bigr\}.
\end{equation}
\end{cor}

\begin{proof}
Write $z_t:=\Phi^0_t(z)$,
$F(t):=t+W(z_t)$, and $A:=\{t\ge0:\ V(z_t)=W(z_t)\}$.  Since $u\equiv1$ is
admissible, $V\le W$ on $\Omega$.

The set $\mathcal T(z)$ is nonempty by
\Cref{prop:existence_opt}(E3) and \Cref{thm:single_switch}.  The uncontrolled
trajectory cannot reach $\mathcal G$ before $V(z)$.  Hence
\Cref{lem:W_regularity}
implies that $F$ is continuous on $[0,V(z))$.  If the uncontrolled trajectory
reaches $\mathcal G$ at $V(z)$, \Cref{lem:transversal_exit} and continuous
dependence imply $W(z_t)\to0$ as $t\uparrow V(z)$.  Thus $F$ is continuous on
$[0,V(z)]$.  On this interval $T^{u_t}(z)=F(t)$, so
$\mathcal T(z)=\{t\in[0,V(z)]:F(t)=V(z)\}$ is compact.  Thus $\tau^*(z)$ is
well defined.

Along the pre-switch arc the optimal control vanishes, so the trajectory of
the optimal pair coincides with $t\mapsto z_t$ on $[0,\tau^*(z)]$, and
\Cref{prop:optimality_principle} gives
\begin{equation}\label{eq:V_along_preswitch}
V(z_t)=V(z)-t,
\qquad 0\le t\le \tau^*(z).
\end{equation}

\emph{$\tau^*(z)\in A$.}  After the switch the optimal control is
$u\equiv1$, so the remaining time to reach $\mathcal G$ from
$z_{\tau^*(z)}$ equals $W(z_{\tau^*(z)})$, whence
$V(z)=\tau^*(z)+W(z_{\tau^*(z)})$.  Combining with
\eqref{eq:V_along_preswitch} at $t=\tau^*(z)$ yields
$W(z_{\tau^*(z)})=V(z)-\tau^*(z)=V(z_{\tau^*(z)})$.

\emph{Minimality.}  Suppose $t\in A$ with $t<\tau^*(z)$.  Then, using
\eqref{eq:V_along_preswitch},
\[
t+W(z_t)=t+V(z_t)=t+\bigl(V(z)-t\bigr)=V(z),
\]
so the single switch control $u_t$ is optimal and $t$ is an optimal
switching time strictly smaller than $\tau^*(z)$, contradicting its
minimality.  Hence $\min A=\tau^*(z)$.
\end{proof}

\begin{rem}
The word \emph{earliest} is needed if
$\tau\mapsto \tau+W(\Phi^0_\tau(z))$ is minimized on a nondegenerate interval.
Every point of that interval is an optimal switching time and lies in $A$.
Equation \eqref{eq:value_matching} selects the left endpoint, as does the
event detection procedure in Method~4.
\end{rem}

\begin{rem}
At points where $V$ is differentiable, the feedback law associated with
\eqref{eq:HJB_formal} satisfies
$u^*(S,I)=\mathbf 1_{\{V_S(S,I)>0\}}$.
Thus the switching time characterization above is consistent with the Pontryagin switching function $\psi=-\lambda_S S$ derived in
\Cref{sec:PMP_switch}.
\end{rem}

\section{Numerical experiments}\label{sec:numerics}

This section illustrates the single switch structure predicted by
\Cref{thm:single_switch}, validates the dynamic programming switching
representation \eqref{eq:V_switch_rep}, compares the observed optimal
eradication time against the explicit upper bounds of
\Cref{thm:finite_time_strong}, and examines numerical cases in which the
condition $G>0$ fails.
All computations are carried out on the invariant set $\Omega$.

\subsection{Nonlinear scenarios and structural conditions}
\label{subsec:numerics_scenarios}
\label{subsec:G_verification}

We consider the three nonlinear configurations of \Cref{tab:scenarios}.
They are phenomenological test families chosen to examine the structural
condition, rather than calibrated models of a particular epidemic.
All three have $\gamma_S=0$, so the experiments isolate changes in $G$
caused by susceptible dependence of transmission.  The theorem also allows
recovery to depend on $S$ and transmission to decrease with $I$, provided
$\beta+I\partial_I\beta>0$.  The present numerical families all satisfy the
stronger condition $\partial_I\beta\ge0$; the decreasing case is illustrated
analytically in \Cref{rem:saturating_incidence}.
Each is examined in two regimes that exhibit
qualitatively different optimal policies.

\begin{table}[!htbp]
\centering
\setlength{\tabcolsep}{3pt}
\caption{The three nonlinear transmission rates and their parameters in both
operating regimes.  In every case the recovery rate is
$\gamma(S,I)=\gamma_0+cI$, and $m=n=1$.  All six parameter sets satisfy the
structural condition $G>0$ on all of $\Omega$
(\Cref{subsec:G_verification}); sets chosen to violate $G>0$ are studied
in \Cref{subsec:sharpness}.}
\label{tab:scenarios}
\begin{tabular}{@{}l l l l@{}}
\hline
\textbf{Scenario} & $\beta(S,I)$ & \textbf{Regime I} & \textbf{Regime II}\\
\hline
\textsc{(A) Saturating} &
$\beta_0\dfrac{1+aI}{1+bS}$ &
$\beta_0=3.0,\ \gamma_0=0.2,$ &
$\beta_0=64,\ \gamma_0=12,$ \\
\textsc{\phantom{(A)} contact} & &
$a=2.5,\ b=1.5,\ c=0.5$ &
$a=2.5,\ b=1.5,\ c=2.0$ \\[4pt]

\textsc{(B) Exponential} &
$\beta_0\exp(\alpha I-\sigma S)$ &
$\beta_0=2.8,\ \gamma_0=0.15,$ &
$\beta_0=40,\ \gamma_0=12,$ \\
\textsc{\phantom{(B)} transmission} & &
$\alpha=3.0,\ \sigma=0.9,\ c=0.4$ &
$\alpha=3.0,\ \sigma=0.3,\ c=2.0$ \\[4pt]

\textsc{(C) Power law} &
\makecell[l]{$\beta_0(1+k_I I)^{m}$\\$\quad{}\times(1-k_S S)^{n}$} &
$\beta_0=2.5,\ \gamma_0=0.12,$ &
$\beta_0=40,\ \gamma_0=12,$ \\
\textsc{\phantom{(C)} heterogeneity} & &
$k_I=2,\ k_S=0.4,\ c=0.35$ &
$k_I=2,\ k_S=0.25,\ c=2.0$ \\ \hline
\end{tabular}

\end{table}

\begin{table}[!htbp]
\centering
\caption{Two dimensionless quantities evaluated at the initial state: the
vaccination to recovery timescale ratio
$\gamma_{\mathrm{init}}$ in \eqref{eq:gamma_init_meaning}, and the reproduction ratio
$R_{\mathrm{eff}}=\beta(S_0,I_0)S_0/\gamma(S_0,I_0)$.  Regime~II has
$\gamma_{\mathrm{init}}\approx12$ and $R_{\mathrm{eff}}\approx2$--$2.5$.
Regime~I has a campaign faster than recovery
($\gamma_{\mathrm{init}}<1$) and an initial reproduction ratio of six
or more.  The corresponding switching behavior is examined in
\Cref{subsec:significance}.}
\label{tab:dimensionless}
\begin{tabular}{l c c c c}
\hline
& \multicolumn{2}{c}{\textbf{Regime I}} & \multicolumn{2}{c}{\textbf{Regime II}}\\
Scenario & $\gamma_{\mathrm{init}}$ & $R_{\mathrm{eff}}$ &
$\gamma_{\mathrm{init}}$ & $R_{\mathrm{eff}}$ \\ \hline
\textsc{(A)} & $0.205$ & $5.98$  & $12$ & $2.12$ \\
\textsc{(B)} & $0.154$ & $7.61$  & $12$ & $2.45$ \\
\textsc{(C)} & $0.124$ & $12.35$ & $12$ & $2.48$ \\ \hline
\end{tabular}

\end{table}

\paragraph{Time units.}
Because the control is normalized to $u\in[0,1]$ and enters as $-uS$, time is
implicitly measured in units of the inverse maximal vaccination rate:
vaccination alone at $u\equiv1$ reduces $S$ by a factor $e$ in one time unit.
Define $\gamma_{\mathrm{init}}:=\gamma(S_0,I_0)$.  This is the dimensionless
ratio
\begin{equation}\label{eq:gamma_init_meaning}
\gamma_{\mathrm{init}}
\;=\;
\frac{\text{time for vaccination alone to reduce $S$ by a factor $e$}}
     {\text{inverse recovery rate at the initial state}} .
\end{equation}
The second dimensionless quantity is the reproduction ratio at the initial
state,
\begin{equation}\label{eq:Reff_def}
R_{\mathrm{eff}}:=\frac{\beta(S_0,I_0)\,S_0}{\gamma(S_0,I_0)}.
\end{equation}
These two ratios influence the qualitative result.  If the initial inverse
recovery rate is five days, $\gamma_{\mathrm{init}}=12$ corresponds to
a vaccination timescale of about two months.  The value
$\gamma_{\mathrm{init}}=0.2$ corresponds to about one day.

We use $R_{\mathrm{eff}}$ rather than $R_0$ because the classical
basic reproduction number \citep{DiekmannHeesterbeekMetz1990} is defined
through invasion near a disease free state.  It can also be defined for
state dependent rates, whereas
here $\beta$ and $\gamma$ are evaluated at the actual initial state
$(S_0,I_0)$ with $I_0>0$, so \eqref{eq:Reff_def} is an initial
\emph{effective} reproduction ratio which moreover varies along the
trajectory.  It is used below only at the initial state.

\paragraph{Regime II: slow campaign and low initial prevalence.}
Here $(S_0,I_0)=(0.99,10^{-5})$, $\varepsilon=10^{-8}$ and
$\gamma_{\mathrm{init}}\approx12$, so
the campaign is twelve times slower than recovery.  The initial prevalence is
of order $10^{-5}$, and $R_{\mathrm{eff}}$ lies between $2.1$ and $2.5$ in the
three scenarios.
This regime is used to study an interior switching time.
\Cref{subsec:significance} examines the effects of $\gamma_0$ and
$R_{\mathrm{eff}}$ separately.

\paragraph{Regime I: fast campaign and higher initial prevalence.}
Here $(S_0,I_0)=(0.99,10^{-2})$, $\varepsilon=10^{-6}$ and
$\gamma_{\mathrm{init}}\approx0.12$ to $0.21$.  By
\eqref{eq:gamma_init_meaning}, the maximal vaccination rate is approximately
five to eight times the initial recovery rate.  This regime
realizes the endpoint case $\tau^*=0$ of \Cref{thm:single_switch}.  It is
interpreted here as a limiting idealization rather than as the primary case.

\begin{rem}
\label{rem:expectation}
\Cref{thm:single_switch} predicts a bang-bang control with at most one
switch, from $u=0$ to $u=1$.  It does \emph{not} predict that the switch is
interior: $\tau^*=0$, i.e.\ immediate full vaccination, is an admissible
instance of the theorem.  Which of the two occurs is determined by the global
minimum of $\phi$.  By \Cref{prop:phi_prime_identity}, the condition
$\partial_S W(S_0,I_0)<0$ rules out $\tau^*=0$, but the converse need not
hold.  The computations give the endpoint case $\tau^*=0$ in Regime~I and an
interior switch in Regime~II.

In the Regime~II examples, delayed vaccination is accompanied by substantial
susceptible depletion through infection before the switch.  In Regime~I,
vaccination acts more rapidly relative to recovery, and immediate
vaccination is the computed optimum.  The two regimes
therefore illustrate the endpoint and interior cases of the theorem.
\Cref{subsec:significance} identifies the parameter range associated with the
interior case.
\end{rem}

\paragraph{Verification of the structural condition.}

For these scenarios we verify the structural condition
$G(S,I)>0$ of \Cref{ass:G_pos} for the models used in the
numerical experiments, thereby determining where the single switch result
(\Cref{thm:single_switch}) applies.  Recall from \eqref{eq:G_def} that
\[
G(S,I)=\beta(S,I)+S\,\partial_S\beta(S,I)-\partial_S\gamma(S,I)
=\partial_S\bigl(\beta(S,I)\,S-\gamma(S,I)\bigr).
\]
In all three scenarios $\gamma$ depends on $I$ only, so $\partial_S\gamma=0$
and $G=\partial_S(\beta S)$.

\paragraph{(A) Saturating contact rate}

\begin{lem}\label{lem:F_scenarioA}
For $\beta(S,I)=\beta_0\dfrac{1+aI}{1+bS}$ and $\gamma(S,I)=\gamma_0+cI$
with $\beta_0>0$ and $a,b\ge 0$, one has
\[
G(S,I)=\frac{\beta_0(1+aI)}{(1+bS)^2}>0\qquad\text{on }\Omega.
\]
Hence the structural inequality $G>0$ holds throughout $\Omega$.
\end{lem}

\begin{proof}
Since $\beta S=\beta_0(1+aI)\,\dfrac{S}{1+bS}$ and
$\partial_S\!\bigl(\tfrac{S}{1+bS}\bigr)=(1+bS)^{-2}$, the formula follows;
positivity is immediate.
\end{proof}

\paragraph{(B) Exponential transmission}

\begin{lem}\label{lem:F_scenarioB}
For $\beta(S,I)=\beta_0 e^{\alpha I-\sigma S}$ and $\gamma(S,I)=\gamma_0+cI$
with $\beta_0>0$ and $\alpha,\sigma\ge 0$, one has
\[
G(S,I)=\beta(S,I)\,(1-\sigma S).
\]
Thus $G>0$ if and only if $\sigma S<1$.  In particular, this inequality
holds on all of $\Omega$ if and only if $\sigma\le 1$, and otherwise only on the
subregion $\{S<1/\sigma\}$.
\end{lem}

\begin{proof}
$\partial_S(\beta S)=\beta_0 e^{\alpha I-\sigma S}(1-\sigma S)=\beta(1-\sigma S)$,
and $\beta>0$.
\end{proof}

\paragraph{(C) Power law heterogeneity}

\begin{lem}\label{lem:F_scenarioC}
For $\beta(S,I)=\beta_0(1+k_I I)^m(1-k_S S)^n$ and $\gamma(S,I)=\gamma_0+cI$
with $\beta_0>0$, $k_I\ge 0$, $0\le k_S<1$, and $m,n\ge 1$, one has
\[
G(S,I)=\beta_0(1+k_I I)^m(1-k_S S)^{n-1}\bigl(1-(n+1)k_S S\bigr).
\]
If $k_S=0$, then $G>0$ on $\Omega$.  If $k_S>0$, then $G>0$ if and only if
$S<1/((n+1)k_S)$.  Thus the structural inequality holds on all of $\Omega$
if and only if $(n+1)k_S\le1$.
\end{lem}

\begin{proof}
$\partial_S\!\bigl[S(1-k_S S)^n\bigr]
=(1-k_S S)^{n-1}\bigl[(1-k_S S)-nk_S S\bigr]
=(1-k_S S)^{n-1}\bigl(1-(n+1)k_S S\bigr)$,
and $(1-k_S S)^{n-1}>0$ since $k_S S<1$ on $\Omega$.
\end{proof}

\begin{rem}\label{rem:scenarios_covered}
The parameters of \Cref{tab:scenarios} satisfy \Cref{ass:main}.  The
structural inequality also holds on all of $\Omega$: scenario~(A) by
\Cref{lem:F_scenarioA}; scenario~(B) because $\sigma=0.9$ and $0.3$ in
Regimes~I and~II, respectively, both at most one; and scenario~(C) because
$(n+1)k_S=0.8$ and $0.5$, respectively, also at most one.  Consequently
\Cref{thm:single_switch} applies directly to all six parameter sets.
\end{rem}

\begin{rem}\label{rem:sharpness_setup}
The thresholds identified in
\Cref{lem:F_scenarioB}--\Cref{lem:F_scenarioC} are explicit, so the
structural condition can also be violated deliberately.  In the numerical
test we use
\[
(\beta_0,\alpha,\sigma,\gamma_0,c)=(48,3,1.5,5,2)
\]
in scenario~(B), for which $G>0$ exactly when $S<2/3$, and
\[
(\beta_0,k_I,k_S,m,n,\gamma_0,c)=(32,2,0.67,1,1,5,2)
\]
in scenario~(C), for which $G>0$ exactly when $S<1/1.34\approx0.746$.
With $S_0=0.99$, both
trajectories start in $\{G<0\}$ and later enter $\{G>0\}$.
\Cref{subsec:sharpness} examines the resulting controls.
\end{rem}

\subsection{Numerical methods}
\label{subsec:compute_tau}
\label{subsec:discretization}

\paragraph{Change of variables.}
As indicated above, we integrate in $(S,L)$ with $L:=\log I$, so that
\eqref{eq:SIR} becomes
\begin{equation}\label{eq:SIR_log}
\dot S=-\beta(S,e^{L})\,S\,e^{L}-u\,S,
\qquad
\dot L=\beta(S,e^{L})\,S-\gamma(S,e^{L}).
\end{equation}
This change of variables has two numerical advantages.  The prevalence spans
several orders of magnitude between the
initial datum and the threshold, and in the original variables the accuracy
near the target would be governed by the absolute tolerance rather than the
relative one; in $L$ the resolution is uniform across decades.  Moreover the
target set becomes the hyperplane $\{L=\log\varepsilon\}$, so that the
eradication time is the root of an affine function of the state.

\paragraph{Forward solver and event detection.}
We integrate \eqref{eq:SIR_log} with the automatic stiff/nonstiff switching
scheme of \citet{Petzold1983} as implemented by \citet{Virtanen2020}, at
relative tolerance $10^{-11}$ and absolute tolerance
$10^{-13}$, and locate the crossing of $\{L=\log\varepsilon\}$ by a terminal
event, that is, by root finding on the dense interpolant of the solution
rather than by interpolating linearly between fixed output times.  This
avoids the error of that linear interpolation; integration and dense
interpolation errors remain and can affect the computed switching time.

\paragraph{Positivity.}
For finite $L$, the reconstruction $I=e^L$ is positive in exact arithmetic.
The differential equation preserves $S>0$, since the $S$ component of
\eqref{eq:SIR_log} vanishes at $S=0$.  This does not give a discrete
positivity guarantee for LSODA.  No state clipping is used in the code.


We compare direct minimization, a Pontryagin consistency check, a derivative
criterion, and a value matching diagnostic.  The latter shares the
minimization used by the first method.

\subsubsection{Method 1: direct minimization}

Using \Cref{thm:switch_rep}, the optimal eradication time satisfies
\[
V(S_0,I_0)=\inf_{\tau\ge 0}\phi(\tau),\qquad
\phi(\tau):=\tau+W(\Phi^0_\tau(S_0,I_0)).
\]
Numerically, we proceed as follows:

\begin{enumerate}[label=(M1-\arabic*), leftmargin=40pt]
\item For a grid of candidate switching times
$\tau_k=k\,\Delta\tau$, $k=0,1,\dots,K$,
simulate the pre-switch flow $\Phi^0_{\tau_k}(S_0,I_0)$ under $u\equiv 0$.
\item From each state $z_k:=\Phi^0_{\tau_k}(S_0,I_0)$, simulate the
post-switch dynamics under $u\equiv 1$ until $I(t)\le\varepsilon$, which
yields $W(z_k)$.
\item Evaluate $\phi(\tau_k)=\tau_k+W(z_k)$ and select
\[
k_*\in\argmin_{0\le k\le K}\phi(\tau_k),
\qquad \tau_{\mathrm{grid}}=\tau_{k_*}.
\]
\end{enumerate}
We refine $\tau_{\mathrm{grid}}$ and other local
minima resolved above the numerical noise floor by bounded Brent
minimization.  All resulting candidates and both endpoints are compared.

In every parameter set satisfying \Cref{ass:G_pos} that we examined, the
global minimum of $\phi$ occurred either at $\tau=0$ or at one interior point
(\Cref{fig:phi}).  \Cref{thm:single_switch} constrains the form of an
optimal control but does not imply unimodality of the switching time
objective.
Where
\Cref{ass:G_pos} fails, minimizing $\phi$ considers only the single switch
family and therefore need not identify an optimal control
(\Cref{subsec:sharpness}).

\subsubsection{Method 2: Pontryagin consistency}
\label{subsubsec:method2}

Alternatively, we solve the coupled state and adjoint system:
forward in time for $(S,I)$, backward in time for the adjoint with
transversality at $T$, and compute the switching function
$\psi(t)=-\lambda_S(t)S(t)$.  When the switch is interior, its unique zero
crossing locates the switch.
We bracket interior sign changes on the adjoint output grid and refine
each root by Brent's method on the dense interpolant.  The terminal zero
$\psi(T)=0$ is excluded.
This provides a direct consistency check for the DPP based procedure above.

\paragraph{A well scaled form of the adjoint.}
Propagating $\lambda_I$ directly is poorly scaled, since
$\lambda_I(T^*)=-1/\dot I(T^*)$ is of order $\varepsilon^{-1}$.  Eliminating
$\lambda_S=-\psi/S$ from \Cref{lem:adjoint_components} and setting
$p:=\lambda_I I$ closes the system in the two quantities that actually carry
the sign information,
\begin{equation}\label{eq:psi_p_system}
\begin{aligned}
\dot\psi &= \beta_S\,S I\,\psi + p\,S\,G,\\
\dot p &= -\psi\,(\beta+\beta_I I)\,I + p\,I\,(\gamma_I-\beta_I S),
\end{aligned}
\end{equation}
with terminal data $\psi(T^*)=0$ and
$p(T^*)=\bigl(\gamma-\beta S\bigr)^{-1}$ evaluated at the exit point, both of
which are finite whenever the exit is transversal
(\Cref{lem:transversal_exit}).  This scaling removes the explicit factor
$\varepsilon^{-1}$ from the terminal value of $\lambda_I$.  Since $I>0$,
$\operatorname{sign}p=\operatorname{sign}\lambda_I$, so the dichotomy of
\Cref{lem:lambdaI_sign} can be read off directly from $p$.

\paragraph{Hamiltonian consistency check.}
The Hamiltonian provides a stringent numerical check of the Pontryagin
conditions.
Along each bang arc, the canonical equations imply that the autonomous
Hamiltonian is constant.  At an interior switch from $0$ to $1$, its jump is
$\psi(\tau)$ because the state and adjoint are continuous.  The terminal
condition gives $H=0$ on the last arc.  Hence, for an interior candidate,
\[
H\equiv0\ \text{on}\ [0,T]
\qquad\Longleftrightarrow\qquad
\psi(\tau)=0,
\]
so $\max_t|H|$ measures the residual in this condition.  For the endpoint
candidate $\tau=0$, there is no jump condition; the sign rule
\eqref{eq:bang_rule} provides the additional check.

\subsubsection{Method 3: derivative criterion}

Method~1 minimizes the function
\[
\phi(\tau)=\tau+W(\Phi^0_\tau(S_0,I_0)).
\]
When $\phi$ is sufficiently smooth, an alternative is to solve the
stationarity condition $\phi'(\tau)=0$ by scalar root finding.

\paragraph{A sensitivity identity for $\phi'(\tau)$.}
Let $z_\tau:=\Phi^0_\tau(z_0)$ with $z_0=(S_0,I_0)$ and denote by
$f_0(z):=f(z,0)$ the pre-switch vector field (i.e., the dynamics under
$u\equiv 0$).  By \Cref{lem:W_regularity}, $\phi$ is continuously
differentiable wherever the exit is transversal, and differentiating yields
\begin{equation}\label{eq:phi_prime}
\phi'(\tau)
=
1+\nabla W(z_\tau)\cdot \frac{d}{d\tau}\Phi^0_\tau(z_0)
=
1+\nabla W(z_\tau)\cdot f_0(z_\tau).
\end{equation}
An interior optimal switching time is therefore a root of
\begin{equation}\label{eq:g_root}
g(\tau):=1+\nabla W(\Phi^0_\tau(z_0))\cdot f_0(\Phi^0_\tau(z_0))=0.
\end{equation}
The roots of $g$ provide candidates for interior optimal switching times.

The following identity simplifies \eqref{eq:g_root}.

\begin{prop}
\label{prop:phi_prime_identity}
Assume \Cref{ass:main} and \Cref{ass:G_pos}.  Then
\begin{equation}\label{eq:phi_prime_identity}
\phi'(\tau)=S(\tau)\,\partial_S W\bigl(\Phi^0_\tau(z_0)\bigr)
\end{equation}
whenever $\Phi^0_\tau(z_0)\notin\mathcal G$.
In particular
\begin{equation}\label{eq:one_sided}
\partial_S W(z_0)<0\ \Longrightarrow\ \tau^*>0,
\qquad\text{equivalently}\qquad
\tau^*=0\ \Longrightarrow\ \partial_S W(z_0)\ge0 .
\end{equation}
\end{prop}

\begin{proof}
The two vector fields differ only through the vaccination term,
\[
f(z,0)-f(z,1)=(S,0),
\qquad\text{i.e.}\qquad
f_0=f(\cdot,1)+(S,0).
\]
By \eqref{eq:W_transport}, $1+\nabla W\cdot f(\cdot,1)=0$.  Substituting into
\eqref{eq:phi_prime} gives
\[
\phi'(\tau)
=1+\nabla W\cdot f_0
=\underbrace{1+\nabla W\cdot f(\cdot,1)}_{=\,0}+S\,\partial_S W
=S\,\partial_S W .
\]
Since $S>0$ on $\Omega$, $\partial_SW(z_0)<0$ gives $\phi'(0)<0$, so $\phi$ is
strictly decreasing at $\tau=0$ and the boundary point $\tau=0$ cannot be a
minimizer; hence $\tau^*>0$.  The second implication in \eqref{eq:one_sided}
is its contrapositive.
\end{proof}

\begin{rem}
\label{rem:no_converse}
The converse of \eqref{eq:one_sided} is not asserted.
The identity \eqref{eq:phi_prime_identity} is local.
If $\partial_SW(z_0)\ge0$, the right derivative of $\phi$ at zero is
nonnegative, but this does not make zero a global minimizer.
Even unimodality or quasiconvexity alone does not make a zero derivative
sufficient for optimality: $h(t)=(t^2-1)^2$ on $t\ge0$ decreases to its
minimum at $t=1$ but satisfies $h'(0)=0$.
Accordingly, $\{\partial_SW=0\}$ in \Cref{subsec:switch_curve} is a
\emph{local switching indicator}, where a short delay has zero first order
effect.  Identifying it with the global switching boundary would require
additional conditions, which are not established here.  Global
minimization of $\phi$ determines the switch.
\end{rem}

\begin{rem}
Identity \eqref{eq:phi_prime_identity} shows that a short delay decreases the
objective to first order exactly when $\partial_SW<0$.  If
$\partial_SW>0$, it increases the objective to first order; when
$\partial_SW=0$, higher order information is needed.  Only the derivative of
$W$ with respect to $S$ is required in Method~3.  The same
derivative is used for the switching indicator in \Cref{subsec:switch_curve}.
This is consistent with the Hamilton--Jacobi feedback
$u^*=\mathbf 1_{\{V_S>0\}}$: on the pre-switch arc, $V$ and $W$ agree at
$\tau^*$, and in both formulations the decision is determined by the
derivative with respect to $S$.
\end{rem}

\paragraph{Numerical evaluation.}
We evaluate the gradient in coordinates $(S,L)$, where $L=\log I$.  Thus
\begin{equation*}
\nabla W\cdot f_0
=\dot S\,\partial_S W+\dot L\,\partial_LW,
\qquad
\dot L=\beta(S,I)S-\gamma(S,I).
\end{equation*}
The derivatives are approximated by centred finite differences in the
interior and second order one-sided differences near the boundary.  The
stencil direction and, when necessary, its step are chosen so that every
evaluation satisfies $S>0$, $I>\varepsilon$, and $S+I\le1$.
Each value of $W$ is computed by simulation under
$u\equiv1$ with the same event detection as in Method~1.  Logarithmic
coordinates give comparable steps over the range of infected fractions used
in the computations.

\paragraph{Root finding.}
We first scan $\tau$ on a coarse grid for sign changes of $g(\tau)$ and apply
Brent's method on each bracketing interval.  We compare $\phi$ at all detected
roots and at both endpoints of the search interval, since a root need not be
a global minimizer.

\begin{rem}
Method~3 is useful when the finite difference estimate of $\nabla W$ is
stable.  If it is sensitive to the time step or event detection, Method~1 can
be used instead.
\end{rem}

\subsubsection{Method 4: value matching}

When approximate value functions $\widehat V$ and $\widehat W$ are available,
the optimal switching time can be computed directly from the value matching
characterization in \Cref{cor:tau_star}.
Recall that the optimal switching time satisfies
\[
\tau^*
=
\min\Bigl\{
t\ge 0:\ 
V(\Phi^0_t(S_0,I_0))=W(\Phi^0_t(S_0,I_0))
\Bigr\}.
\]
Define the exact residual along the trajectory before the switch by
\begin{equation*}
r(t)
:=
V(\Phi^0_t(S_0,I_0))
-
W(\Phi^0_t(S_0,I_0)).
\end{equation*}
Applying \Cref{thm:switch_rep} at the state $z_t:=\Phi^0_t(z_0)$ and using the
semigroup property $\Phi^0_\sigma(z_t)=\Phi^0_{t+\sigma}(z_0)$ gives
$V(z_t)=-t+\inf_{s\ge t}\phi(s)$, while $W(z_t)=\phi(t)-t$; hence the exact
residual is
\begin{equation}\label{eq:residual_identity}
r(t)=\inf_{s\ge t}\phi(s)-\phi(t)\ \le\ 0 .
\end{equation}
It is strictly negative for $t<\tau^*$ and vanishes at $t=\tau^*$.  The
residual need not vanish identically on $[\tau^*,\infty)$:
by \eqref{eq:residual_identity} that would require $\phi$ to be nondecreasing
after $\tau^*$, which is a global shape property we have not established (cf.\
\Cref{rem:no_converse}).

For approximations $\widehat V$ and $\widehat W$, set
\[
\widehat r(t)
:=
\widehat V(\Phi^0_t(S_0,I_0))
-
\widehat W(\Phi^0_t(S_0,I_0)).
\]
With a tolerance $\delta>0$, Method~4 defines
\begin{equation*}
\widehat\tau_\delta:=\min\{t\ge0:\ \widehat r(t)\ge -\delta\}.
\end{equation*}

\paragraph{Implementation.}
In the reported computations, Method~4 is a value matching diagnostic, not
an independent approximation of $V$.  We use the minimum $m$ and minimizer
$\tau_m$ estimated by Method~1 and evaluate $m-\phi(t)$ for $t\le\tau_m$.
The refined minimizer is inserted into the time grid, so the search contains
a point satisfying the detection condition.  We locate the first grid
interval whose endpoints satisfy
$\widehat r(t_{n-1})+\delta<0\le\widehat r(t_n)+\delta$ and apply bisection
while preserving this bracket.  The returned endpoint is checked against
$\widehat r\ge-\delta$.  We use $\delta=10^{-9}$ and record the residual
alongside the switching time.  Methods~1 and~4 therefore share the
objective evaluations and are not independent validations of optimality.

\paragraph{Conditioning.}
Suppose the earliest minimizer is interior, $\phi$ is twice differentiable
at $\tau^*$, and the minimum is nondegenerate,
\begin{equation*}
\tau^*>0
\qquad\text{and}\qquad
\phi''(\tau^*)>0 .
\end{equation*}
Then for $t$ in a left neighbourhood of $\tau^*$,
\[
r(t)=-\bigl(\phi(t)-\phi(\tau^*)\bigr)
=-\tfrac12\phi''(\tau^*)\,(t-\tau^*)^2+o\bigl((t-\tau^*)^2\bigr),
\]
so the residual approaches zero quadratically from the left.  To state an
error estimate, suppose $\widehat r$ is continuous and
$\|\widehat r-r\|_{\infty}\le E$ on $[0,\tau^*]$, and choose $\delta>E$.
Then detection occurs no later than $\tau^*$, since $r(\tau^*)=0$.
On every compact interval strictly before $\tau^*$, the continuous residual
is bounded away from zero.  Thus, for sufficiently small $\delta+E$, the
first detection lies in the left neighbourhood where the expansion applies.
At that point $-r\le\delta+E$, and consequently
\[
0\le\tau^*-\widehat\tau_\delta
\le C\sqrt{\delta+E}.
\]
In particular, $E=O(\eta)$ and $\delta=O(\eta)$ with $\delta>E$ give
$O(\sqrt\eta)$ localization.  A bound on $E$ alone does not guarantee
detection.  A discretized search must also resolve the detection
neighbourhood; otherwise grid error can dominate.  Contact of order $2k$
gives exponent $1/(2k)$ under the corresponding assumptions.  A boundary
minimum is checked directly.  By contrast, a uniform error in $\phi'$
produces linear root sensitivity when $\phi''(\tau^*)>0$.

\begin{rem}\label{rem:learned_values}
The approximations $\widehat V$ and $\widehat W$ may, for example, be
represented by neural networks as in \citet{KimKimKim2025}.  Method~4
requires both functions.  Method~3 requires only $W$, which solves the linear
equation
$1+\nabla W\cdot f(\cdot,1)=0$, in which no minimization over $u$ appears.
Thus Method~3 does not require the nonsmooth Hamiltonian in
\eqref{eq:HJB_explicit}.

Concretely, let $z_\tau:=\Phi^0_\tau(z_0)$.  Since $z_\tau$ is the flow of the
ODE $\dot z=f_0(z)$, it satisfies $\frac{d}{d\tau}z_\tau=f_0(z_\tau)$.
If $\widehat V$ and $\widehat W$ are represented by differentiable neural
networks, then by the chain rule
\[
\frac{d}{d\tau}\bigl(\widehat V(z_\tau)-\widehat W(z_\tau)\bigr)
=
\nabla(\widehat V-\widehat W)(z_\tau)\cdot f_0(z_\tau).
\]
For Method~3, the corresponding approximation is
\[
\widehat g(\tau)
=
1+\nabla \widehat W(z_\tau)\cdot f_0(z_\tau).
\]
The gradients $\nabla\widehat V$ and $\nabla\widehat W$ are obtained by
automatic differentiation.
\end{rem}

\subsection{Validation of the single switch structure}
\label{subsec:validate_single_switch}
\label{subsec:bound_compare}

\Cref{tab:main_results} collects the computed switching and eradication
times and the a priori bound for all six parameter sets, and
\Cref{tab:pmp} reports the Pontryagin diagnostics.  The two regimes exhibit
the two alternatives in \Cref{thm:single_switch}.


\begin{table}[!htbp]
\centering
\setlength{\tabcolsep}{3pt}
\caption{Computed switching and eradication times, the eradication time
$W(z_0)$ under immediate full vaccination, the relative saving, and the
a priori bound of \Cref{thm:finite_time_strong} with the explicit extrema in
\Cref{subsec:scenario_extrema}.  Methods~1 and~3 agree to all digits shown.
Method~4 uses the minimum estimated by Method~1 and reports the first
detected point satisfying the value matching tolerance $\delta=10^{-9}$.
Its earlier switching time reflects this tolerance, as explained in
\Cref{subsec:compute_tau}.}
\label{tab:main_results}
\begin{tabular}{c c c c c c c c c c}
\hline
Regime & Sc. & $\tau^*$ (M1) & $\tau^*$ (M3) & $\tau^*$ (M4) & $T^*$ &
$W(z_0)$ & saving (\%) & $T_{\mathrm{ub}}$ & $T_{\mathrm{ub}}/T^*$\\ \hline
I & A & 0.00000 & 0.00000 & 0.00000 & 54.63250 & 54.63250 & 0.00 & 579.5 & 10.6 \\
I & B & 0.00000 & 0.00000 & 0.00000 & 72.96159 & 72.96159 & 0.00 & 5036.3 & 69.0 \\
I & C & 0.00000 & 0.00000 & 0.00000 & 91.79864 & 91.79864 & 0.00 & 754.6 & 8.2 \\
II & A & 0.68037 & 0.68037 & 0.68032 & 2.24961 & 2.94137 & 23.52 & 138.2 & 61.4 \\
II & B & 0.51578 & 0.51578 & 0.51575 & 2.07456 & 2.77648 & 25.28 & 653.9 & 315.2 \\
II & C & 0.57975 & 0.57975 & 0.57971 & 2.13389 & 2.75072 & 22.43 & 63.3 & 29.6 \\
\hline
\end{tabular}

\end{table}

\paragraph{Both branches of the dichotomy occur.}
In Regime~I the computed optimum is the corner $\tau^*=0$: the switching function has
no zero on $[0,T^*)$ and $\lambda_I>0$ throughout, corresponding to the first
branch of \Cref{lem:lambdaI_sign} and the case $\Theta\le0$ in the proof of
\Cref{thm:single_switch}.  In Regime~II the switch is interior,
$\tau^*$ between $0.52$ and $0.68$, $\psi$ has exactly one interior zero, and
$\lambda_I$ changes sign exactly once.  This is the second branch.  Both are
instances of the theorem.
\Cref{fig:trajectories} shows the Regime~II trajectories and control
profiles, and \Cref{fig:psi} the switching function together with
$p=\lambda_I I$, which has the same sign as $\lambda_I$.

\begin{figure}[!htbp]
\centering
\includegraphics[width=\textwidth]{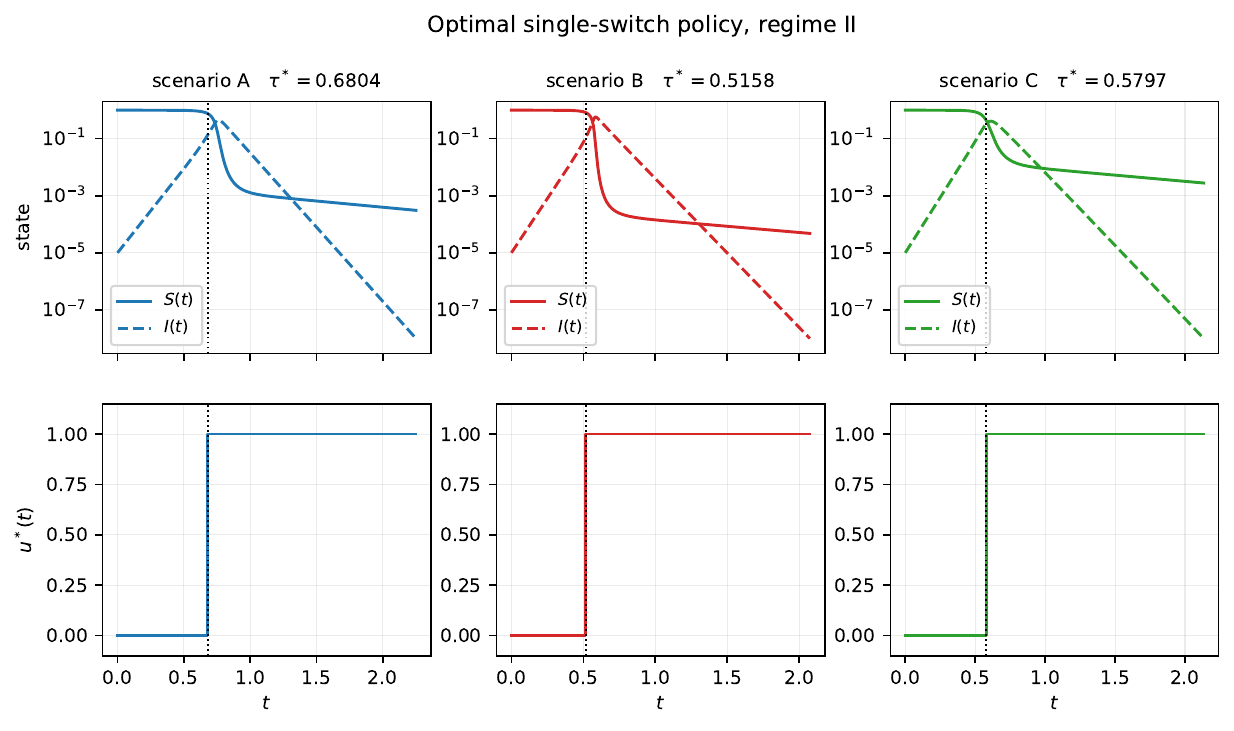}
\caption{Regime~II.  Top: the computed time-optimal trajectories $S(t)$ (solid) and $I(t)$
(dashed) on a logarithmic scale; the dotted vertical line marks the switching
time $\tau^*$.  Bottom: the corresponding control, a single jump from $0$
to~$1$.  Before the switch, the susceptible fraction decreases through
infection; vaccination is then applied at full intensity.}
\label{fig:trajectories}
\end{figure}

\paragraph{Hamiltonian consistency.}
At an interior switch, \Cref{subsubsec:method2} shows that continuity of the
Hamiltonian is equivalent to $\psi(\tau)=0$.  In every case
$\max_t|H|$ is of order $10^{-9}$ to $10^{-8}$ at the computed $\tau^*$ and is larger by four
to nine orders of magnitude at $\tau^*+0.3$ (\Cref{tab:pmp}).  The sign rule
also holds in all cases, including the endpoint cases in Regime~I.  These are
numerical checks of the Pontryagin conditions.


\begin{table}[!htbp]
\centering
\setlength{\tabcolsep}{3pt}
\setlength{\tabcolsep}{3pt}
\caption{Verification of the Pontryagin conditions (Method~2).  In Regime~I
the switching function has no interior zero and $\lambda_I>0$ throughout, which is the
first branch of the dichotomy of \Cref{lem:lambdaI_sign}; in Regime~II it has
exactly one interior zero and $\lambda_I$ changes sign exactly once, the second branch.
The terminal zero $\psi(T^*)=0$ is excluded from the counts.
The Hamiltonian residual is of order $10^{-9}$ to $10^{-8}$ at the computed $\tau^*$ and is
larger by four to nine orders of magnitude at $\tau^*+0.3$.  The sign rule
also holds in every case.}
\label{tab:pmp}
\begin{tabular}{c c c c c c c c}
\hline
Regime & Sc. & \makecell{interior zeros\\of $\psi$} & location &
\makecell{rule\\holds} & \makecell{$\max|H|$\\at $\tau^*$} &
\makecell{$\max|H|$\\at $\tau^*\!+\!0.3$} &
\makecell{sign changes\\of $\lambda_I$}\\ \hline
I & A & 0 & n/a & yes & 2.285e-09 & 5.319e+00 & 0 \\
I & B & 0 & n/a & yes & 2.822e-09 & 6.810e+00 & 0 \\
I & C & 0 & n/a & yes & 1.001e-08 & 9.670e+00 & 0 \\
II & A & 1 & 0.68037 & yes & 4.983e-09 & 3.819e-03 & 1 \\
II & B & 1 & 0.51578 & yes & 6.507e-09 & 2.064e-04 & 1 \\
II & C & 1 & 0.57975 & yes & 1.410e-08 & 2.941e-02 & 1 \\
\hline
\end{tabular}

\end{table}

\begin{figure}[!htbp]
\centering
\includegraphics[width=\textwidth]{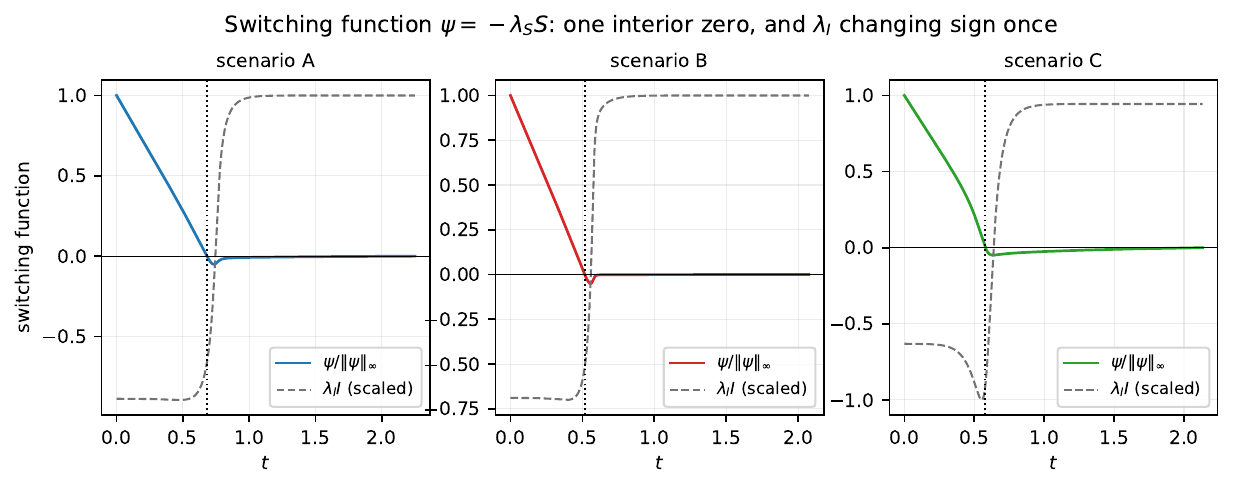}
\caption{Regime~II: the switching function $\psi=-\lambda_SS$, normalized,
together with the rescaled costate $p=\lambda_II$ of
\eqref{eq:psi_p_system}.  In each scenario $\psi$ has exactly one interior zero, at
the dotted line $t=\tau^*$, and $\lambda_I$ changes sign exactly once, from
negative to positive, as predicted by \Cref{lem:lambdaI_sign}.}
\label{fig:psi}
\end{figure}

\paragraph{Agreement and conditioning.}
Methods~1 and~3 agree to all digits reported.  Method~4 detects a point
slightly before the minimum, as expected for the tolerance condition
$\widehat r\ge-\delta$.  Every returned point satisfies that condition.
\Cref{tab:method4_noise} examines sensitivity by adding independent uniform
perturbations in $[-\eta,\eta]$ to the sampled residual and using
$\delta=2\eta$.  We use scenario~(A) of Regime~II with
$(S_0,I_0)=(0.99,10^{-5})$, $\varepsilon=10^{-8}$, and $2001$ points on
$[0,2]$, giving grid spacing $10^{-3}$.  Sixty realizations are used at
each noise level, with NumPy's default random generator and seed zero.
This bounded perturbation experiment illustrates the conditioning estimate;
it does not establish an asymptotic convergence rate.  Method~3 requires
one further caution:
$\phi'$ may vanish at several points, and a root is only a critical point.
The candidate roots and both endpoints must therefore be compared.
On nearly flat portions of the objective, finite difference noise can also
produce sign changes that do not reliably locate stationary points.


\begin{table}[!htbp]
\centering
\setlength{\tabcolsep}{3pt}
\caption{Value matching under independent uniform perturbations in
$[-\eta,\eta]$ of the sampled residual, with tolerance $\delta=2\eta$.
Scenario~(A) of Regime~II is used with $\varepsilon=10^{-8}$ and the nominal
initial state.  Each row reports the median over $60$ realizations on
$[0,2]$ with grid spacing $10^{-3}$ and random seed zero.
The last column describes sensitivity relative to the local
square root scale; finite-grid measurements do not establish an asymptotic
convergence rate.}
\label{tab:method4_noise}
\begin{tabular}{c c c c}
\hline
$\eta$ & $\sqrt\eta$ & median $|\tau_{\mathrm{M4}}-\tau^*|$ &
ratio to $\sqrt\eta$\\ \hline
1e-05 & 0.0032 & 0.005369 & 1.698 \\
1e-04 & 0.0100 & 0.020869 & 2.087 \\
1e-03 & 0.0316 & 0.064369 & 2.035 \\
1e-02 & 0.1000 & 0.187369 & 1.874 \\
1e-01 & 0.3162 & 0.470869 & 1.489 \\
\hline
\end{tabular}

\end{table}

\begin{figure}[!htbp]
\centering
\includegraphics[width=\textwidth]{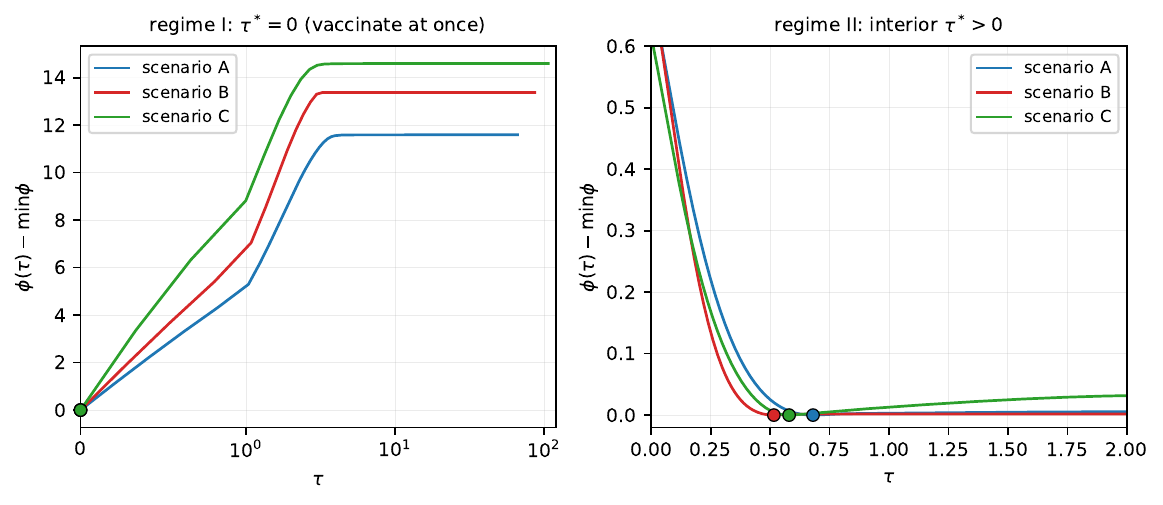}
\caption{The objective $\phi(\tau)=\tau+W(\Phi^0_\tau(z_0))$ after subtracting
its sampled minimum, with the optimizer marked.  Left: Regime~I, where the global
minimum occurs at the boundary $\tau^*=0$ (note the symmetric log horizontal
scale).  Right: Regime~II, where $\phi$ has a unique interior minimum over the
displayed interval.}
\label{fig:phi}
\end{figure}

\paragraph{Convergence.}
\Cref{tab:convergence} refines the $\tau$ grid by a factor of sixteen; the
switching time is stable to six decimal places.  The eradication threshold
mainly changes the duration of the terminal decay; its effect on the switching
time is small in the computations.


\begin{table}[!htbp]
\centering
\setlength{\tabcolsep}{3pt}
\caption{Convergence of $\tau^*$ (Regime~II) under refinement of the
$\tau$ grid used by Method~1.  The value is stable to six decimal places.}
\label{tab:convergence}
\begin{tabular}{c c c c c c }
\hline
Scenario & $n=61$ & $n=121$ & $n=241$ & $n=481$ & $n=961$\\ \hline
A & 0.680368976 & 0.680369015 & 0.680369031 & 0.680369028 & 0.680369074 \\
B & 0.515784153 & 0.515784136 & 0.515784139 & 0.515784166 & 0.515784155 \\
C & 0.579747701 & 0.579747717 & 0.579747639 & 0.579747670 & 0.579747663 \\
\hline
\end{tabular}

\end{table}

\paragraph{Upper bound versus observed eradication time.}

The last two columns of \Cref{tab:main_results} compare the a priori bound of
\Cref{thm:finite_time_strong}, evaluated with the explicit extrema of
\Cref{subsec:scenario_extrema}, against the computed optimum.  The bound
holds in every case, and is conservative by a factor between $8$ and $315$.
The variation is explained in part by scenario~(B), where
$\beta_{\max}=\beta_0e^{\alpha}$ substantially exceeds the transmission rate
because the exponential bound is attained only at $I=1$, a state the
trajectory never approaches.  The bound is intended as a guaranteed,
inexpensive simulation horizon rather than as an accurate quantitative
estimate.  The extrema are already exact on the triangle.  A tighter time
estimate would need a smaller reachable set or a sharper comparison that
retains more of the trajectory dependence.

\subsection{Improvement over immediate full vaccination}
\label{subsec:significance}
\label{subsec:switch_curve}

The single switch theorem characterizes the form of the optimal policy.  To
quantify the reduction in eradication time relative to immediate maximal
vaccination, we compare $V(z_0)$ with $W(z_0)$.

\paragraph{Frequency of a positive improvement.}
In Regime~I the computed optimum is immediate full vaccination, so the saving
is zero.  To examine a broader range, we swept the three families over
the grid
\[
\begin{gathered}
\beta_0\in\{8,16,32,64,128\},
\qquad
\gamma_0\in\{0.5,1.5,3,6,12\},\\
I_0\in\{3\!\cdot\!10^{-6},10^{-5},3\!\cdot\!10^{-5},10^{-4}\},
\end{gathered}
\]
with $S_0=0.99$ and $\varepsilon=10^{-8}$ throughout, and with the remaining
shape parameters taking the values $b\in\{0.5,1.5,4,10\}$ and
$c\in\{0.5,2\}$ in family~(A); $\sigma\in\{0.1,0.3,0.6,0.9\}$ and
$\alpha\in\{1,3\}$ in family~(B); $k_S\in\{0.1,0.25,0.4,0.5\}$ and
$k_I\in\{1,2\}$ in family~(C).  We fix $a=2.5$ in family~(A), $c=2$
in families~(B) and~(C), and $m=n=1$ in family~(C).
All $\SweepN$ parameter sets satisfy
$G>0$ on $\Omega$.  Each case is evaluated using a grid of $121$ switching
times followed by local refinement, and its parameters and outcome are
recorded.  The search ends at $\min\{T^{u\equiv0},60\}$.
In every case with $T^{u\equiv0}>60$, immediate full vaccination reaches
the target before time $60$, so a later switch cannot improve that feasible
policy.  Failed computations are recorded separately and are not assigned
zero saving.  The reported sweep contains no failed cases.
The saving is computed as zero for $\SweepZero\%$ of the sets, using a
numerical zero threshold of $10^{-6}$ percentage points.  Both the median
and the 75th percentile are zero at this precision.  These
proportions apply only to the prescribed grid and would change under a
different grid or weighting.  A saving above $1\%$ occurs in $\SweepAboveOne\%$ of the
parameter sets, and a saving above $10\%$ occurs in $\SweepAboveTen\%$.  Within a
single scenario, the distribution over initial states is similar but less
concentrated: \Cref{fig:saving_map} shows the relative saving over $\Omega$, and
roughly a third of the sampled admissible initial states have saving above
$1\%$, concentrated at large $S_0$ and small $I_0$.  This is a sample
fraction, not an area fraction in $(S,I)$: the grid is uniform in $S$ and
logarithmic in $I$.

\begin{figure}[!htbp]
\centering
\includegraphics[width=\textwidth]{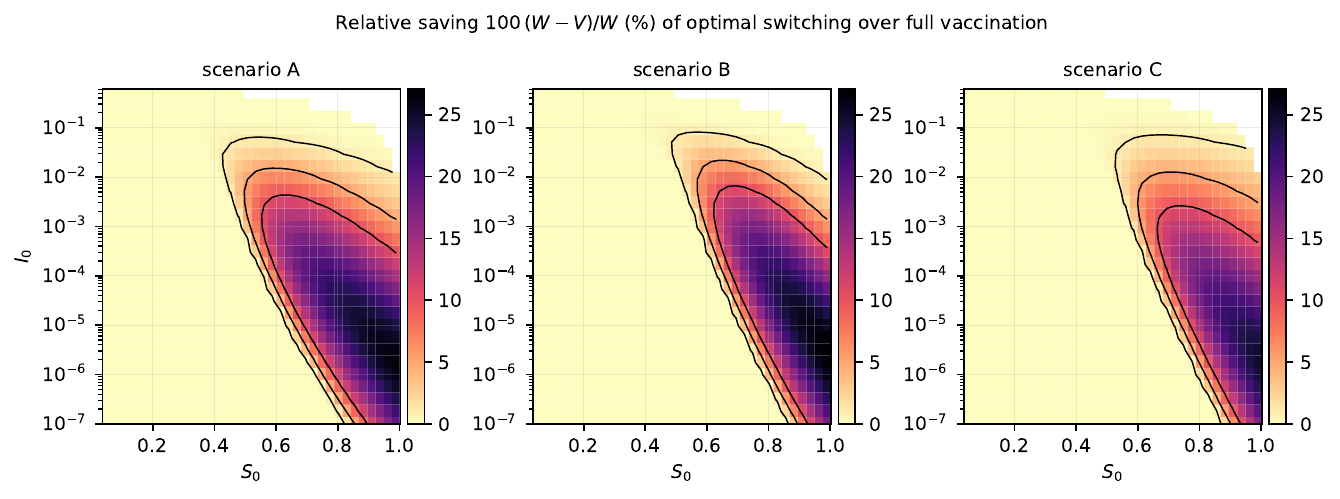}
\caption{Relative saving $100\,(W-V)/W$ of the computed policy over immediate
full vaccination, on the plotted subset of $\Omega$ in Regime~II.  Contours
are drawn at $1\%$, $5\%$ and $10\%$.  The saving is concentrated at large
$S_0$ and small $I_0$.}
\label{fig:saving_map}
\end{figure}

\paragraph{Magnitude of the improvement.}
The largest saving found anywhere in the admissible sweep is $\SweepMax\%$, and
the three Regime~II parameter sets, chosen near that maximum, realize
$23.5\%$, $25.3\%$ and $22.4\%$, respectively.  The largest absolute saving in that sweep was
$\SweepMaxAbs$ time units, which is less than one maximal vaccination time unit.

\paragraph{Dependence on the eradication threshold.}
For scenario~(A) of Regime~II with $(S_0,I_0)=(0.99,10^{-5})$, varying
$\varepsilon$ from $10^{-6}$ to $10^{-12}$ changes the absolute saving
$W(z_0)-V(z_0)$ by about a fifth, from $0.630$ to $0.755$, while the relative
saving falls from $25.2\%$ to $20.0\%$, because $W(z_0)$ itself grows like
$\log(1/\varepsilon)$.  A relative saving quoted for an eradication time
problem is therefore only meaningful alongside the threshold that produced
it, and we report both throughout.

\paragraph{Comparison with no vaccination.}
The reference policy matters.  In Regime~II the optimal policy reduces
eradication time relative to no vaccination by only $0.26\%$, $0.09\%$ and
$1.49\%$ in scenarios~(A), (B) and (C), respectively
(\Cref{tab:baselines}).  The larger reductions relative to immediate full
vaccination therefore reflect the objective's preference for rapid
susceptible depletion through infection.  They do not indicate a comparable
improvement over the duration of the uncontrolled epidemic.  The burden
measures below quantify this distinction.


\begin{table}[!htbp]
\centering
\setlength{\tabcolsep}{3pt}
\caption{Comparison with immediate full vaccination and no vaccination
in Regime~II.  The reduction depends strongly on the reference policy.}
\label{tab:baselines}
\begin{tabular}{c c c c c c}
\hline
& \multicolumn{3}{c}{eradication time} & \multicolumn{2}{c}{reduction (\%)}\\
Sc. & optimal & immediate & none & vs. immediate & vs. none\\ \hline
A & 2.249606 & 2.941370 & 2.255442 & 23.52 & 0.26 \\
B & 2.074559 & 2.776479 & 2.076363 & 25.28 & 0.09 \\
C & 2.133886 & 2.750724 & 2.166176 & 22.43 & 1.49 \\
\hline
\end{tabular}

\end{table}

\paragraph{Dependence on $R_{\mathrm{eff}}$.}
Sweeping $\beta_0$ at fixed
$\gamma_0$, so that the vaccination timescale is held fixed while
$R_{\mathrm{eff}}=\beta(z_0)S_0/\gamma(z_0)$ varies, the sampled interval with
saving above $0.5\%$ has a lower endpoint between approximately $1.7$ and
$2.0$ across the three families, and an upper endpoint between $3.3$ and
$5.0$.  These are sampled level crossings, not exact thresholds for zero
saving.  The
saving reaches a maximum between $22\%$ and $25\%$ for
$R_{\mathrm{eff}}$ between approximately $2.0$ and $2.5$; see
\Cref{fig:saving_R0}.  The three scenarios agree closely on the height of the
maximum and differ moderately in the width of the interval.  When
$R_{\mathrm{eff}}$ is close to one, the uncontrolled dynamics decrease the
susceptible fraction slowly.  For large $R_{\mathrm{eff}}$, a delay produces a
higher prevalence peak.  The improvement occurs at intermediate
$R_{\mathrm{eff}}$, where the susceptible fraction can decrease sufficiently
rapidly before vaccination.  By \Cref{prop:phi_prime_identity}, the local
criterion for this behavior is $\partial_SW<0$.

\begin{figure}[!htbp]
\centering
\includegraphics[width=\textwidth]{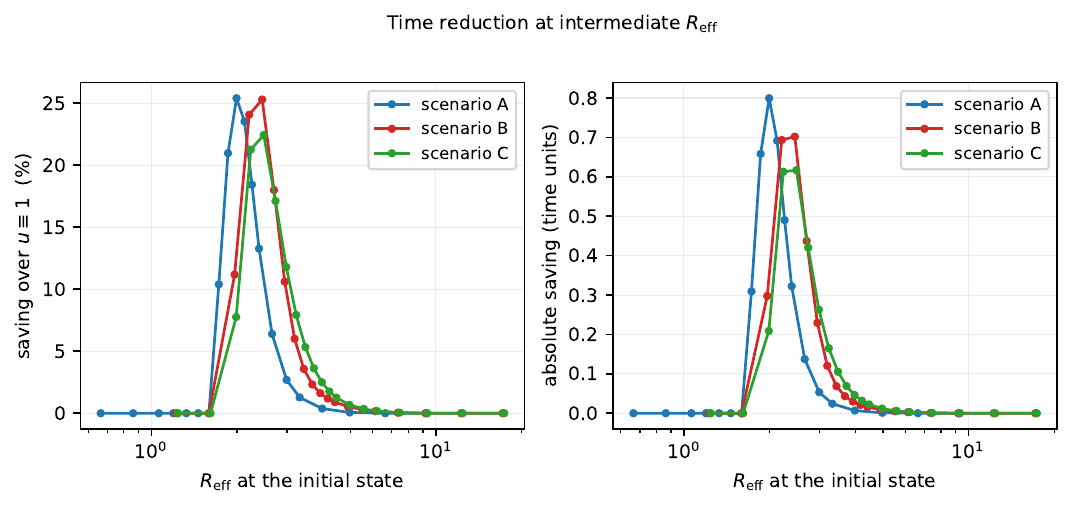}
\caption{Saving of the computed policy over immediate full vaccination as a
function of the reproduction number at the initial state, relative (left) and
absolute (right).  $R_{\mathrm{eff}}$ is varied by sweeping $\beta_0$ at fixed $\gamma_0$,
so the vaccination timescale is held fixed.  A positive saving occurs only
over an intermediate range of $R_{\mathrm{eff}}$.}
\label{fig:saving_R0}
\end{figure}

\paragraph{Effect on epidemiological burden.}
\label{par:burden}
The preceding comparison uses only the time required to reach
$I=\varepsilon$ and does not account for peak prevalence, cumulative
prevalence, or cumulative incidence.  Define
\[
C_{\rm inf}:=\int_0^T\beta(S,I)SI\,\de t,
\qquad
C_{\rm vac}:=\int_0^T uS\,\de t.
\]
These quantities satisfy
$S_0-S(T)=C_{\rm inf}+C_{\rm vac}$.
During the delay, susceptible individuals leave the susceptible compartment
through infection rather than vaccination.  \Cref{tab:burden} quantifies the
consequence for the three Regime~II
scenarios.  A reduction in eradication time between $22\%$ and $25\%$ raises the peak
prevalence $\max_tI$ by a factor of $12.5$ to $19.3$.  The peak is about
$3\%$ under immediate full vaccination and between $41\%$ and $57\%$ under
the computed time-optimal delay.  The cumulative prevalence
$\int_0^TI\,\de t$ increases
by a factor of $4.3$ to $4.6$, and cumulative incidence increases by a factor
of $4.5$ to $4.8$.  Under the delayed policy, $C_{\rm inf}$ is between $0.949$
and $0.962$, while $C_{\rm vac}$ is between $0.025$ and $0.040$.  Under
immediate full vaccination, the corresponding ranges are $0.198$ to $0.213$
and $0.746$ to $0.760$.  The reduction in eradication time is therefore
obtained mainly by susceptible depletion through infection.


\begin{table}[!htbp]
\centering
\setlength{\tabcolsep}{3pt}
\setlength{\tabcolsep}{3pt}
\caption{Eradication time and burden measures in Regime~II.  The increase
columns are ratios of the minimum time policy to immediate full vaccination.
Cumulative incidence is $\int_0^T\beta(S,I)SI\,\de t$, and the vaccinated
fraction is $\int_0^T uS\,\de t$.}
\label{tab:burden}
\begin{tabular}{c c c c c c c c c}
\hline
& \multicolumn{2}{c}{eradication time} & \multicolumn{2}{c}{increase} &
\multicolumn{2}{c}{cumulative incidence} &
\multicolumn{2}{c}{vaccinated fraction}\\
Sc. & minimum time & $u\equiv1$ & peak & $\int_0^T I\,\de t$ &
minimum time & $u\equiv1$ & minimum time & $u\equiv1$\\ \hline
A & 2.2496 & 2.9414 & 14.4$\times$ & 4.3$\times$ & 0.949 & 0.213 & 0.040 & 0.755 \\
B & 2.0746 & 2.7765 & 19.3$\times$ & 4.6$\times$ & 0.954 & 0.198 & 0.036 & 0.760 \\
C & 2.1339 & 2.7507 & 12.5$\times$ & 4.4$\times$ & 0.962 & 0.212 & 0.025 & 0.746 \\
\hline
\end{tabular}

\end{table}

\begin{rem}
\label{rem:not_a_recommendation}
Minimum eradication time is a well-posed criterion studied here and elsewhere
\citep{HyndEradication2021,BolzoniTimeOptimal2017,JangKim2025}.  It does not,
however, account for cumulative incidence, peak prevalence, hospitalization,
or mortality.  The delayed vaccination policies in this paper are therefore
solutions to the stated minimum time problem and should not be interpreted as
public health recommendations.  The values in \Cref{tab:burden} show that the
minimum time and burden objectives may differ substantially.  A formulation
that includes several objectives, such as peak or cumulative prevalence, is
outside the scope of the present study.
\end{rem}

\begin{rem}
\label{rem:significance_reading}
In the prescribed sweep, immediate maximal vaccination is the computed
optimum or lies within $1\%$ of it for most parameter sets.  A positive
delay occurs mainly at low initial prevalence, when vaccination is slow
relative to recovery and $R_{\mathrm{eff}}$ is in an intermediate range.  The
structural contribution of \Cref{thm:single_switch} is independent of the
magnitude of this improvement: it reduces the control problem to the
representation \eqref{eq:V_switch_rep} in one variable.  The parameter sweep
quantifies the subset of the prescribed grid in which the interior case
occurs.
\end{rem}

\paragraph{A local switching indicator.}

To relate the observed savings to the derivative criterion, we compute a
local switching indicator over the state space.
Direct evaluation of $\tau^*(S_0,I_0)$ would require a scalar
minimization at every initial state.  \Cref{prop:phi_prime_identity} instead
permits evaluation of the scalar field
\[
z\ \longmapsto\ 1+\nabla W(z)\cdot f_0(z)\ =\ S\,\partial_SW(z),
\]
which requires only a finite difference stencil for $W$ at each grid point.

\Cref{fig:switching_curve} displays this field for the three Regime~II
scenarios.  The region where delay is locally beneficial occurs at large $S$
and small to moderate $I$, and is separated from regions with small $S$ or
large $I$.  The initial state used throughout
\Cref{subsec:validate_single_switch} lies inside it.  The
qualitative shape is the same in all three scenarios.  This is consistent
with a dependence on the SIR dynamics rather than on the particular
nonlinearity chosen for $\beta$.  Comparing with
\Cref{fig:saving_map}, the region with $\partial_SW<0$ is consistent with the
region of positive saving observed in the numerical experiments.

\begin{figure}[!htbp]
\centering
\includegraphics[width=\textwidth]{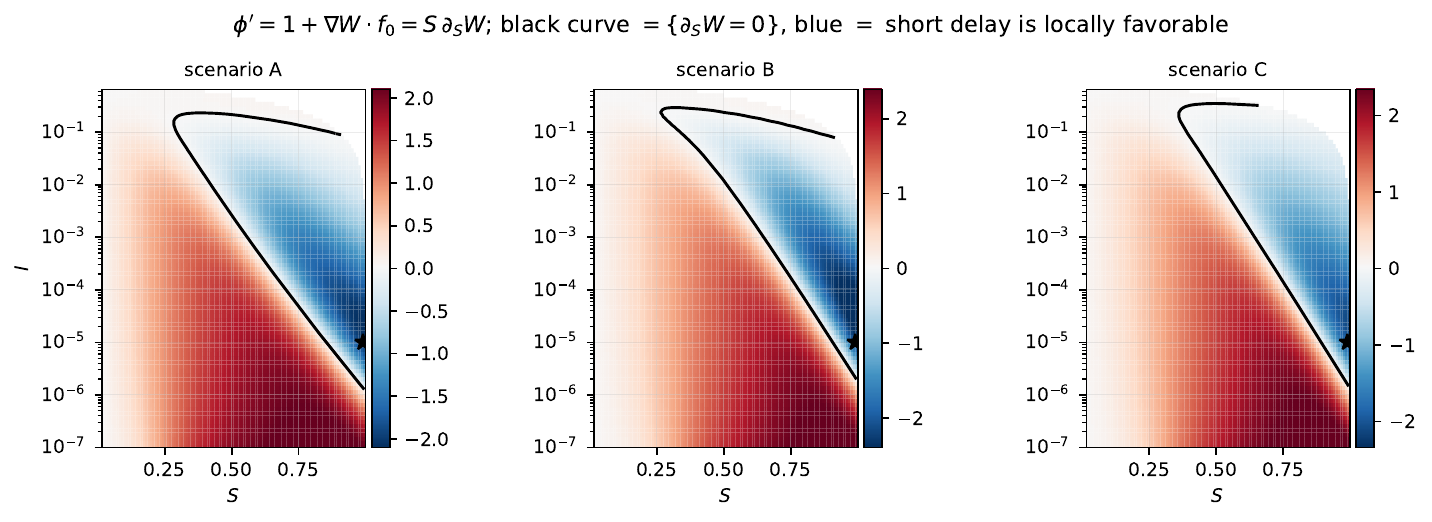}
\caption{The field $\phi'=S\,\partial_SW$ over $\Omega$ for the Regime~II
scenarios.  The black curve is its zero level set, which by
\Cref{prop:phi_prime_identity} separates the regions in which delay is locally
beneficial ($\partial_SW<0$, blue) or not locally beneficial
($\partial_SW>0$, red).  The star marks the initial state
$(S_0,I_0)=(0.99,10^{-5})$.}
\label{fig:switching_curve}
\end{figure}

\begin{rem}
At differentiability points of $V$, the Hamilton--Jacobi feedback reads
$u^*=\mathbf 1_{\{V_S>0\}}$.  The computed curve can therefore be compared
with $\{V_S=0\}$ in the region preceding a switch.  The computation itself
uses only $W$.
\end{rem}

\subsection{Numerical examples outside the structural regime}
\label{subsec:sharpness}

\Cref{ass:G_pos} is sufficient for the single switch conclusion, but the
proof of \Cref{thm:single_switch} does not show that it is necessary: the
crossing direction of $\psi$ at a zero depends on $\lambda_I G$,
and a trajectory that spends its initial phase in $\{G<0\}$ before entering
$\{G>0\}$ is not excluded a priori from switching more than once.  We
present numerical results exhibiting this behavior.

\paragraph{Design of the experiment.}
We use the violating parameter sets of \Cref{rem:sharpness_setup} together
with two comparison parameter sets satisfying $G>0$.  They use
$(\beta_0,\alpha,\sigma,\gamma_0,c)=(20,3,0.3,3,2)$ in scenario~(B) and
$(\beta_0,k_I,k_S,m,n,\gamma_0,c)=(20,2,0.2,1,1,3,2)$ in scenario~(C).
All four tests use $(S_0,I_0)=(0.99,10^{-5})$ and
$\varepsilon=10^{-8}$.  For each parameter set we compute two
quantities: the best computed policy within the single switch family
$\{u_\tau\}_{\tau\ge0}$, obtained by a fine grid search and local refinement,
and the best computed policy in the family that starts with $u=1$ and has two
switch times.  The latter is obtained by multistart local optimization and
contains the single switch family as boundary cases.  If it offers no
improvement, the single switch result remains consistent with the
computations.  Within numerical accuracy, a computed improvement demonstrates
that restriction to a single switch is suboptimal when \Cref{ass:G_pos} is
violated.

We repeat the two violating cases under six numerical configurations.  These
vary the ODE solver tolerances, the grid used for the single switch search,
the number of starting points in the multiple switch search, and the random
seed.  The complete results are summarized in
\Cref{tab:sharpness_robustness}.

\paragraph{Outcome.}
The results are collected in \Cref{tab:sharpness}.  Where $G>0$ holds, the
two switch family reduces to $0\to1$ through a boundary or cancelling pair of
switches, with eradication times agreeing to the precision reported in
\Cref{tab:sharpness}.
Where $G>0$ fails, the best policy found in the two switch family has the
profile
\[
1\ \longrightarrow\ 0\ \longrightarrow\ 1,
\]
with two switches.  Its eradication time is smaller than the computed minimum
over the single switch family by $2.16\%$ in scenario~(B) and $3.85\%$ in
scenario~(C).  The switching function computed along that single switch
trajectory also fails the sign condition in
\eqref{eq:bang_rule}.

\Cref{fig:sharpness} gives diagnostics along the computed two switch
profiles.  In both cases the two control switches agree with the two zeros of
the switching function, and its sign gives the computed control through
\eqref{eq:bang_rule}.  The trajectory starts in the region $G<0$ and later
enters the region $G>0$.


\begin{table}[!htbp]
\centering
\setlength{\tabcolsep}{3pt}
\setlength{\tabcolsep}{3pt}
\caption{Numerical comparison inside and outside the regime $G>0$.  Where
$G>0$ fails, the best computed profile $1\to0\to1$ improves on the computed
minimum over the single switch family, and the latter fails the Pontryagin
sign condition.  Where $G>0$ holds, the best computed profile is the single
switch $0\to1$ predicted by \Cref{thm:single_switch}.  The last column counts
interior zeros along the single switch candidate, not the multiple switch
profile.  The terminal zero is excluded.}
\label{tab:sharpness}
\setlength{\tabcolsep}{1.8pt}
\begin{tabular}{c c c c c c c c}
\hline
Parameters & \makecell{grid\\minimum of $G$} & \makecell{best single\\switch} &
\makecell{best multiple\\switches} & \makecell{improvement\\(\%)} &
\makecell{computed\\profile} & \makecell{\#\\switches} &
\makecell{zeros of $\psi$\\(single)}\\ \hline
B-viol & -5.84 & 5.020048 & 4.911549 & 2.1613 & $1\to0\to1$ & 2 & 2 \\
C-viol & -10.77 & 5.211797 & 5.011356 & 3.8459 & $1\to0\to1$ & 2 & 2 \\
B-ok & 10.39 & 6.781955 & 6.781955 & 0.0000 & $0\to1$ & 1 & 1 \\
C-ok & 12.02 & 6.768091 & 6.768091 & 0.0000 & $0\to1$ & 1 & 1 \\
\hline
\end{tabular}

\end{table}

\paragraph{Scope of the numerical conclusion.}
At the selected parameter values, the computed two switch control has a
smaller eradication time than the computed minimum over the single switch
family.  The latter was obtained from a fine grid search with local
refinement, and the observed gap exceeds the estimated discretization error
by several orders of magnitude.  Across all six configurations in
\Cref{tab:sharpness_robustness}, both gaps remain positive and the computed
profile remains $1\to0\to1$.  These results provide reproducible numerical
counterexamples to the unrestricted single switch conclusion.
As a separate integration check, we reintegrate the reported single and
two switch controls with DOP853 at relative tolerance $2\times10^{-12}$
and absolute tolerance $2\times10^{-14}$.  Both improvements persist.
This check compares fixed candidates and does not perform a new global
optimization.

This comparison establishes, within numerical accuracy, that restriction to
the single switch family is suboptimal at the selected parameter values.  It
does not determine the structure of the unrestricted optimum.  Outside
$\{G>0\}$ the proof of \Cref{lem:no_singular} no longer applies, so singular
arcs are not excluded.  We therefore do not assert that the true optimum has
exactly two switches or that $G>0$ is necessary for the single switch
property.


\begin{table}[!htbp]
\centering\footnotesize
\setlength{\tabcolsep}{3pt}
\caption{Robustness of the two examples outside $G>0$.  The best single and
multiple columns give the baseline values.  The minimum difference is over
six configurations, and maximum variation is the largest range of either
eradication time.  The relative solver tolerance takes the values $10^{-9}$,
$10^{-11}$ and $10^{-12}$, the single switch grid contains $250$, $500$ or
$1000$ points, and the multiple switch search uses $14$ or $30$ random starts
with seeds $0$, $1$ and $7$.  Every run returns $1\to0\to1$ and a positive
difference.}
\label{tab:sharpness_robustness}
\begin{tabular}{c c c c c c c c}
\hline
Parameters & \makecell{best\\single} & \makecell{best\\multiple} &
\makecell{minimum\\difference} & \makecell{maximum\\variation} &
\makecell{first\\switch} & \makecell{second\\switch} &
\makecell{reproduced\\runs} \\ \hline
B-viol & 5.02004826 & 4.91154936 & 0.10849890 & 2.8e-09 & 0.387541 & 1.091088 & 6/6 \\
C-viol & 5.21179702 & 5.01135650 & 0.20044050 & 2.0e-08 & 0.257312 & 1.300453 & 6/6 \\
\hline
\end{tabular}

\end{table}

\FloatBarrier

\begin{figure}[!htbp]
\centering
\includegraphics[width=0.9\textwidth]{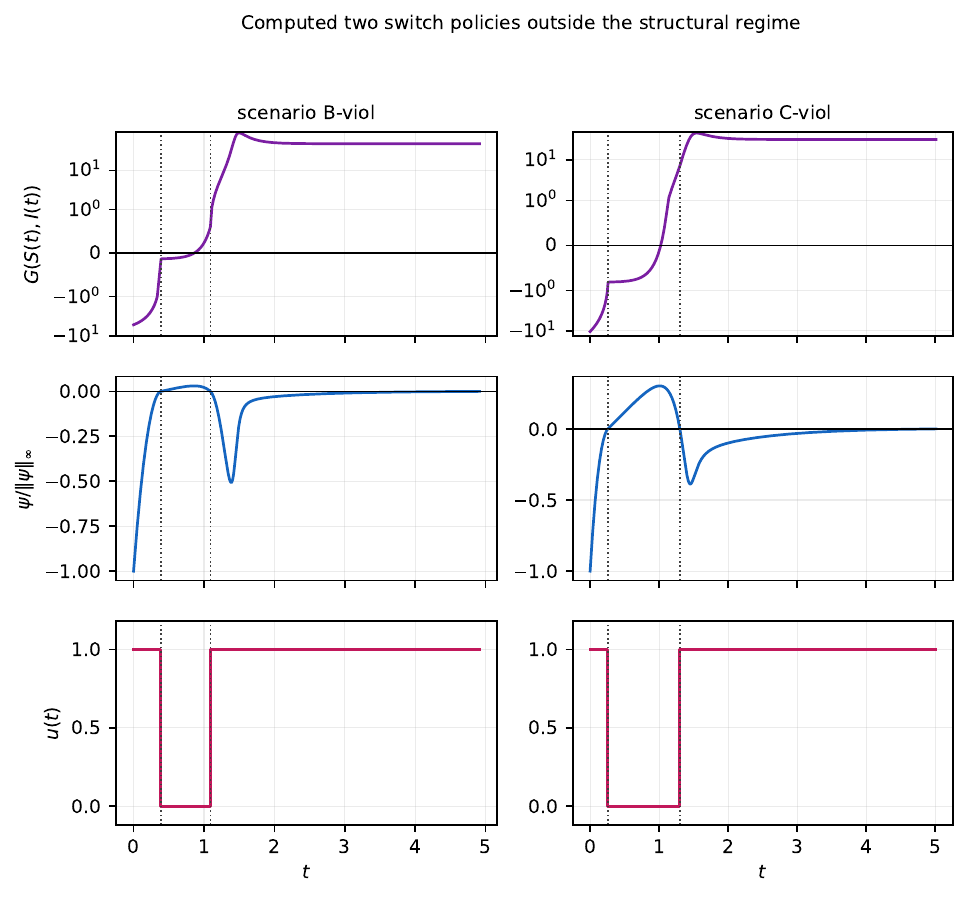}
\caption{Diagnostics along the best computed profile in the multiple switch
family for the two parameter sets that violate $G>0$.  From top to bottom,
the panels show $G(S(t),I(t))$, the normalized switching function, and the
control.  The dotted lines mark the computed switch times.  The profile is
$1\to0\to1$, and the switches coincide with changes in the sign of the
switching function.}
\label{fig:sharpness}
\end{figure}

\begin{rem}
The computed profile in the violating regime vaccinates immediately,
interrupts vaccination, and resumes it later.  Since
$G=\partial_S(\beta S-\gamma)<0$, decreasing the susceptible fraction
increases the instantaneous growth rate of the infected compartment at fixed
$I$.  After a short initial vaccination phase, the computed control
interrupts vaccination while the trajectory passes through the remainder of
$\{G<0\}$ and resumes it after entering $\{G>0\}$.  The initial vaccination
phase reflects the global nature of the time optimal problem.  The crossing
direction depends on the product $\lambda_I G$, not on $G$ alone.
\end{rem}

\FloatBarrier

\section{Conclusion}

We studied minimum time vaccination in an SIR model whose transmission and
recovery rates depend on state.  Under $\beta+I\partial_I\beta>0$ and
$G=\partial_S(\beta S-\gamma)>0$, every time-optimal control is bang-bang and
has at most one switch from $u=0$ to $u=1$.  The proof combines the sign of
$G$, the susceptible sensitivity of per capita infected growth, with the
sign structure of the infected costate.  The incidence condition allows
transmission coefficients that decrease with prevalence, including the
saturating incidence family in \Cref{rem:saturating_incidence}.
For selected parameter sets
violating $G>0$, a computed two switch control has a shorter eradication time
than the computed minimum over the $0\to1$ single switch family.  This provides a
reproducible numerical counterexample to the unrestricted single switch
conclusion.  It does not determine the number of switches in an unrestricted
optimum.

Existence of an optimal control and explicit eradication time bounds support
this analysis and provide finite simulation horizons.  Under the structural
criterion, the scalar switching time representation and the identity
$\phi'(\tau)=S\,\partial_SW$ extend the characterizations of
\citet{HyndEradication2021} to the present nonlinear setting.

In the prescribed parameter sweep, $\SweepAboveOne\%$ of the admissible parameter sets
yield an improvement over immediate full vaccination above $1\%$, and the
largest relative reduction in
eradication time is $\SweepMax\%$.  Immediate maximal vaccination is the computed
optimum for $\SweepZero\%$ of the sets.  In the three Regime~II examples, reductions
between $22\%$ and $25\%$ are accompanied by increases in peak prevalence by
factors between $12.5$ and $19.3$, and increases in cumulative incidence by
factors between $4.5$ and $4.8$.  The corresponding improvements over no
vaccination are about $0.09\%$ to $1.49\%$.  These comparisons quantify the
duration and burden tradeoff identified by \citet{BolzoniTimeOptimal2017}
for the selected nonlinear families.

Possible extensions include models with additional compartments, partial
observability, interventions that depend on time or face resource constraints,
and uncertain transmission and recovery rates.  Formulations with several
objectives could include peak or cumulative prevalence.

\section*{Statements and Declarations}

\subsection*{Funding}
Yeoneung Kim is supported by the National Research Foundation of Korea (NRF)
grants funded by the Korea government (MSIT) (RS-2023-00219980,
RS-2026-25490291) and Yonsei University.

\subsection*{Competing interests}
The authors declare that they have no competing interests.

\subsection*{Code and data availability}

The code and numerical data for
\Cref{tab:main_results,tab:pmp,tab:convergence,tab:sharpness,tab:sharpness_robustness,tab:burden,tab:method4_noise,tab:baselines}
and \Cref{fig:trajectories,fig:phi,fig:psi,fig:sharpness,fig:saving_map,fig:saving_R0,fig:switching_curve}
are available at \url{https://github.com/yeoneung/hetero_sir}
(commit \href{https://github.com/yeoneung/hetero_sir/tree/d0ce69009dfefccbdd81d8eb022e42e7d2f719ea}{\texttt{d0ce690}}).
They include the model and solvers,
all experiment configurations and random seeds, the full parameter sweep,
and scripts generating the figures, tables, and summary statistics.
Run metadata record software versions, numerical settings, source hashes,
and completion status.  The sweep records every attempted case.
Regression tests and independent integrator checks are included.

\subsection*{AI assistance}
The authors used ChatGPT to assist with manuscript preparation.
The authors reviewed and edited the content and take full
responsibility for the final manuscript.

\bibliographystyle{plainnat}
\bibliography{ref}

\end{document}